\documentclass[aop,MSNbibl,nameyear,dvips]{arximspdf}
\usepackage{graphicx}
\doi{10.1214/11-AOP724}
\volume{41}
\issue{3B}
\pubyear{2013}
\firstpage{1900}
\lastpage{1937}

\makeatletter
\newtheorem{theorem}{Theorem}
\newtheorem{lemma}[theorem]{Lemma}
\newtheorem{proposition}[theorem]{Proposition}
\newtheorem{corollary}[theorem]{Corollary}

\newproclaim{remark}[theorem]{Remark}
\newcommand{\PP}{\mathbf P}
\newcommand{\eps}{\varepsilon}
\newcommand{\Z}{{\mathbb Z}}
\newcommand{\R}{{\mathbb R}}
\newcommand{\ev}{\mathbf{E}}
\newcommand{\pr}{\mathbf P}
\newcommand{\one}{{\mathbf1}}
\newcommand{\supp}{\operatorname{supp}}
\newcommand{\phx}{\Phi_{\mathbf{h}, \mathbf{x}}}
\newcommand{\mhx}{\mu_{\mathbf{h}, \mathbf{x}}}
\newcommand{\Rhxde}{\mathcal R(\BS{h}, \BS{x}, \gd, \eps)}
\newcommand{\gd}{\delta}
\newcommand{\gl}{\lambda}
\newcommand{\di}{\,d}
\newcommand{\BS}[1]{\mathbf{#1}}
\newcommand{\C}[1]{\mathcal{#1}}
\newcommand{\D}[1]{\mathbb{#1}}
\newcommand{\ol}[1]{\overline{#1}}
\newcommand{\ul}[1]{\underline{#1}}
\newcommand{\Apl}{A^{+}}

\newcommand{\BMrestricted}{a}
\newcommand{\restrictionNoFloor}{b}
\newcommand{\restrictionFloor}{c}
\newcommand{\lastCost}{d}
\newcommand{\mesh}{\operatorname{mesh}}
\newcommand{\eqref}[1]{(\ref{#1})}
\newcommand{\varlimsup}{\overline{\lim}}
\newcommand{\fracf}[2]{({#1})/({#2})}
\makeatother

\begin{document}
\begin{frontmatter}

\title{Patterns in Sinai's walk}
\runtitle{Patterns in Sinai's walk}

\begin{aug}
\author[A]{\fnms{Dimitris} \snm{Cheliotis}\corref{}\thanksref{au1}\ead[label=e1]{dcheliotis@math.uoa.gr}}
\and
\author[B]{\fnms{B\'alint} \snm{Vir\'ag}\thanksref{au2}\ead[label=e2]{balint@math.toronto.edu}\ead[label=u2,url]{www.math.toronto.edu/\textasciitilde balint}}
\runauthor{D. Cheliotis and B. Vir\'ag}
\affiliation{University of Athens and University of Toronto}
\address[A]{Department of Mathematics\\
University of Athens\\
Panepistimiopolis\\
15784 Athens \\
Greece\\
\printead{e1}} 
\address[B]{Departments
of Mathematics and Statistics\\
University of Toronto\\
Toronto, Ontario\\ Canada, M5S 3G3\\
\printead{e2}\\
\printead{u2}}
\end{aug}
\thankstext{au1}{Supported in part by the DFG-NWO Bilateral Research
Group ``Mathematical Models from Physics and Biology.''}
\thankstext{au2}{Supported by an NSERC Discovery Accelerator Grant and
the Canada Research Chair Program.}

\received{\smonth{2} \syear{2011}}
\revised{\smonth{9} \syear{2011}}

%
\begin{abstract}
Sinai's random walk in random environment shows interesting
patterns on the exponential time scale. We characterize the
patterns that appear on infinitely many time scales after
appropriate rescaling (a functional law of iterated
logarithm). The curious rate function captures the
difference between one-sided and two-sided behavior.
\end{abstract}

%

\end{frontmatter}

\section{Introduction}\label{sec1}

For every integer $k$, pick $p_k$ independently from a fixed
probability measure on $[0,1]$. Then, keeping the $p_k $'s fixed,
consider a nearest-neighbor random walk $S(n)$ on $\D{Z}$, with
$S(0)=0$ and with
probabilities $p_k, 1-p_k$ of going right and left from $k$,
respectively. This model, introduced by \citet{CR}, is the most
well-studied model of motion in random medium.

We will assume that the random variables $p_1, (1-p_1)$ have
some finite negative moment. The walk $S(n)$, pictured in
Figure~\ref{SinaiWalk}, is recurrent exactly when $\log
\frac{1-p_1}{p_1}$ has mean zero; see \citet{SO}. The
graph of the walk seems much more confined than the
ordinary random walk. Indeed, when $\log\frac{1-p_1}{p_1}$
has finite and positive variance as well, the typical value
of $|S(n)|$ is of the order of $\log^2 n$, much less than
the usual $\sqrt{n}$ for simple random walk; see
\citet{SI}. The walk in this regime is called \textit{Sinai's walk}.

The logarithmic behavior of $|S(n)|$ suggests that we may get a
more enlightening picture by considering $S(n)$ on an exponential
time scale, namely the process $t\mapsto S(e^t)$ with the argument
rounded down to the next integer. Figure~\ref{SinaiWalkLog} shows
that the walk tends to get trapped by the environment. Indeed, the
stationary measure for $S(n)$ is given by the exponential of a
function with increments $\log\frac{1-p_k}{p_k}$,\vadjust{\goodbreak} that is, a
random walk on $\mathbb Z$. So at distance $n$ there are regions
with stationary measure as large as $e^{\sqrt n}$, in which $S(n)$
gets trapped for a long time.

\begin{figure}

\includegraphics{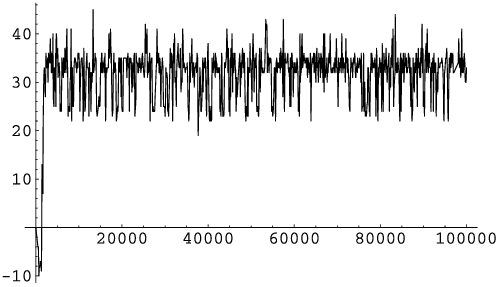}

\caption{Path of Sinai's walk $S(n)$.}\label{SinaiWalk}
\end{figure}

\begin{figure}[b]

\includegraphics{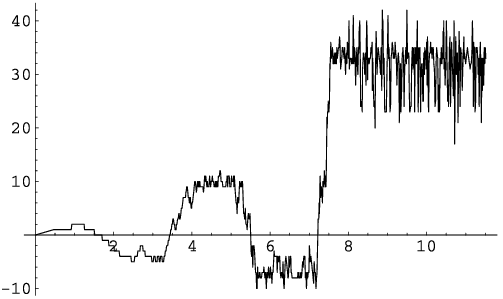}

\caption{The same path in exponential time,
$S(e^t)$.}\label{SinaiWalkLog}
\end{figure}

The pattern we see in Figure~\ref{SinaiWalkLog} suggests a natural
question: what patterns can we get that way? The main goal of this
paper is to answer a mathematically precise version of this
question. We consider rescaled versions of the path of $S(e^t)$
given~by
\[
\frac{S( e^{at})}{a^2\log\log a} , \qquad t\ge0,
\]
and ask what are the possible limit points of the graph of
this process as $a\to\infty$. For this, a topology on
graphs has to be specified. As we will see, the spatial
scaling factor $a^2\log\log a$ is needed to ensure that the
answer to our question is nontrivial.

Figure~\ref{SinaiWalkLog} suggests that we should consider a
topology much weaker than the usual uniform-on-compacts
convergence of functions: the process shows too many oscillations
on this scale, and we do not even expect a function in the limit.
Instead, we consider the graph occupation measure, and we view it as an
element of the space of measures on $\D{R}^+ \times
\D{R}$.
For a measurable $\varphi\dvtx  \D{R}^+ \to
\D{R}$, its \textit{graph occupation measure} is given by
\[
m(\varphi)(A):=\operatorname{Leb}\{s\ge0 \dvtx  (s, \varphi(s))\in A\}
\]
for $A\subset[0, \infty)\times\D{R}$ Borel set.
Note that $m(\varphi)( \cdot \times\D{R})$ is the Lebesgue
measure.

The space of Borel measures on $\D{R}^+ \times\D{R}$ is equipped
with the
topology of local weak convergence.
Then consider the following subset of that space:
\[
\mathcal M= \left\{\mu\dvtx \quad
\begin{array}{l}
\mu( \cdot \times\D{R}) \mbox{ is the Lebesgue measure, } \\
\exists f,g \ge0 \mbox{ nondecreasing s.t. }
\supp(\mu) \subset\overline{\operatorname{graph}(f)}\cup
\overline{\operatorname{graph}(-g)}
\end{array}
 \right\}.
\]
For $\mu\in\mathcal M$, let $f_\mu$ and $g_\mu$ denote the
unique minimal left-continuous choice of $f,g$ in the above
definition.
Now $\mu$ projects to Lebesgue measure on
$\D{R}^+$, so $\mu$ restricted to the upper and lower half
planes project to a partition of Lebesgue measure. Let
$s_{\mu+}, s_{\mu-}\in[0,\infty]$ denote the supremum of
the support of these projections, respectively. Let
%
\begin{equation} \label{rate}
 \quad I(\mu):= \frac{\pi^2}{2}\int_0^{s_{\mu+}}\frac{1}{t^2}\di(
f_\mu+ g_\mu)(t)+ \frac{\pi^2}{8}\int_{s_{\mu+}}^\infty
\frac{1}{t^2}\di g_\mu(t) \qquad  \mbox{if } s_{\mu-}=\infty,
\end{equation}
and if $s_{\mu-}<\infty$, in \eqref{rate} we exchange $f_\mu,
g_\mu$ and replace $s_{\mu+}$ with $s_{\mu-}$. Note the striking difference
between the parts with $\pi^2/8$ and $\pi^2/2$ coefficients---we will see that it is harder to be supported on the
graph of two functions than on a single one.

\begin{theorem}\label{FunLIL}
With probability 1, the $a\to\infty$ limit points of the graph
occupation measures of the rescaled walk
\[
\frac{S(e^{at})}{a^2\log\log a} , \qquad t\ge0,
\]
constitute the set
\[
\C{K}:=\{ \mu\in\C{M}\dvtx  I(\mu)\le1\}.
\]
Also, there is at least one limit point along every sequence
$a_n\to\infty$.
\end{theorem}

Our result is the analogue of Strassen's functional law of
iterated logarithm for ordinary random walks [\citet{ST}]. This is
often stated in terms of the rescaled process restricted to a
\emph{finite} interval; such results easily follow from the full version.
We also prove a version for the \textit{Brox diffusion}, the
continuous version of Sinai's walk; see Theorem
\ref{LILforDiffusion} in Section~\ref{MotLIL}. As discussed there,
Theorem~\ref{FunLIL} extends to general environments that are
close to Brownian motion. Our results are in agreement with
Theorems 1.3, 8.1 of \citet{HS} about the one-point law,
\[\limsup_{a\to\infty}\frac{S(e^{a})}{a^2 \log\log
a}=\frac{8}{\pi^2}.\vadjust{\goodbreak}
\]

There are many ways in which Sinai's walk and the Brox diffusion
are determined by their environment; see, for example, the results
of \citet{HU}, quoted as Theorem~\ref{Hu Theorem} in the present
paper. In particular, the location of $S(n)$ is well predicted by
$x(\log n)$, where $x$ is the process of wells for the
environment. For the Brox diffusion, on the process level, the
first author, \citet{CH}, showed that after some large random time,
the path of the process $x(\log t)$ is close to that of the
most-favorite-point process of the diffusion at time $t$.

In Section~\ref{stepSection}, we give a precise description of the
process of wells $x_B$ for the environment defined by two-sided
Brownian motion $B$. Informally, consider the graph of $B$ as a
vessel in which water is poured gradually from the positive $y$
axis. The water forms several increasing and merging puddles, also
called wells. Then $x_B(h)$ is the $x$-coordinate of the bottom of
the first-created well with depth at least~$h$.

Our law of iterated logarithm is based on a similar theorem for
the process of wells. This, in turn, is based on a large
deviation principle for this process.\vspace*{-2pt}

\begin{theorem}\label{bLDP}
The family of the laws of $\{m (x_B/M )\dvtx  M>0\}$,
as \mbox{$M\to\infty$}, satisfies a large deviation principle on
$\C{M}$ with speed $M$ and good rate function $I$.\vspace*{-2pt}
\end{theorem}

Interestingly, it is easier to avoid creating deep wells on
one just of the axes than on both; this will be apparent
from the proof of the theorem. This is the main reason for
the two different factors $\pi^2/8$ and $\pi^2/2$.

In the flavor of the applications in \citet{ST}, we prove
the following simple result about weighted integrals of
$S(\cdot)$ along a geometric time scale.\vspace*{-2pt}

\begin{corollary}\label{c.application}
For $r\ge0$,
\[
\limsup_{a\to\infty} \frac{1}{a^2\log\log a}\int_0^1 t^r
S(e^{at})\, dt=\frac{4}{\pi^2} \biggl(\frac{2}{r+3} \biggr)^{\fracf
{r+3}{r+1}}.\vspace*{-2pt}
\]
\end{corollary}

\begin{remark}
There is a connection between our results and Chung's Law of
iterated logarithm, which concerns the liminf behavior of the
running maximum of random walk $\{S_n\}$ with increments of zero
mean and variance~1. It states [see \citet{JP}]
\[
\liminf_{n\to\infty} \log\log n \frac{\max_{1\le i\le n}
S_i^2}{n}=\frac{\pi^2}{8}.
\]
Note the presence of the constant $\pi^2/8$ also here. The reason
is that the only way Sinai's walk will take an unusually large
value at a given time is if in a large interval the environment
does not create large wells, which could delay the walk. This, in
effect, confines the
environment to a small interval for a long time. In this sense,
our result is related to Wichura's theorem, a~functional law of
iterated logarithm for small values of the running absolute
maximum of Brownian motion; see \citet{MU}.\vadjust{\goodbreak}
\end{remark}

\subsection*{Orientation}  The structure of the paper is
as follows. The first goal is to prove Theorem~\ref{bLDP}. Thus, Section
\ref{UpperBoundSection} contains the large deviations upper
bound, and Section~\ref{specialSets} contains the lower
bound. Section~\ref{LDPSection} combines these two results
and exponential tightness to derive Theorem~\ref{bLDP}.
Section~\ref{EnvLIL} contains the proof of a functional law
of the iterated logarithm for the environment, that is, for
the family $(a^2\log\log a)^{-1}x_B(a \cdot ), a>e$. This
is combined in Section~\ref{MotLIL} with a localization
result to transfer the law to the motion. In Section
\ref{BrownianComputations} we estimate the probability that
Brownian motion stays in certain sets for large intervals
of time. The last section contains topological lemmas
needed in Sections~\ref{UpperBoundSection},~\ref{specialSets} and
\ref{EnvLIL}.

\section{The process of wells in the environment}\label{sec2} \label{stepSection}

Let $f\dvtx \R\to\R$ be a continuous function. In Section~\ref{sec1}, we
introduced the process of wells by the following informal
definition.

Consider the graph of $f$ as a vessel in which water is poured
gradually from the positive $y$ axis. The water forms several
increasing and merging puddles, also called wells. Then $x_f(h)$
is the $x$-coordinate of the bottom of the first-created well with depth
at least $h$.

We now proceed to give a more detailed definition. For each point
$x_0$ of local minimum for $f$, there are intervals $[a,c]$
containing $x_0$ with the property that $f(x_0)$ is the minimum
value of $f$ in $[a, c]$ and $f(a), f(c)$ are the maximum values
of $f$ on the intervals $[a, x_0], [x_0, c]$, respectively. Let
$[a_{x_0}, c_{x_0}]$ be the maximal such interval.
We call
$f|[a_{x_0}, c_{x_0}]$ the \textit{well} of $x_0$ and the number
\[
\min\{f(a_{x_0})-f(x_0), f(c_{x_0})-f(x_0)\}
\]
the \textit{depth}
of the well. We order wells by inclusion.

For $h>0$, if there is a minimal well of depth at least $h$ containing
zero in its domain, we define $x_f(h)$ to be the smallest point in the
domain of the well where $f$ attains its minimum value on the well. If
there is none, we let $x_f(h)=0$. Finally, we let $x_f(0)=0$.

For almost all two-sided Brownian paths $B$, for all $h>0$,
there is a unique point where $B$ attains its minimum in
the minimal well of depth at least $h$ containing 0, and
there is such a well. For such paths, $x_B$ is a
left-continuous step function. Moreover, $x_B$ has the
following monotonicity property: if $h_1<h_2$ and
$x_B(h_1), x_B(h_2)$ have the same sign, then
$|x_B(h_1)|\le|x_B(h_2)|$.

Finally, $x_B$ inherits a scaling property from Brownian motion,
namely for $a>0,$ %
%
\begin{equation}\label{xBScaling}
(x_B(a s))_{s\ge0}\stackrel{\C{L}}{=}(a^2x_B(s))_{s\ge0}.
\end{equation}
The first step toward the proof of Theorem~\ref{bLDP} is to study
the behavior of the function $x_B$. More precisely, we will try to
understand the probability that $x_B$ is close to a particular
step function.\vadjust{\goodbreak}

\begin{figure}[b]

\includegraphics{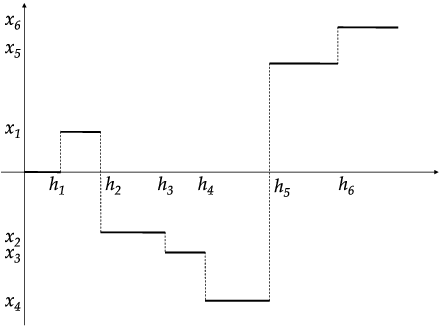}

\caption{The step function $\Phi_{\mathbf{h}, \mathbf{x}}$.} \label
{stepFunction}
\end{figure}

Toward this end, let $\C{S}$ be the set of all pairs of finite
sequences
\[
(\BS{h}, \BS{x}) \qquad\mbox{where }
\BS{h}:=(h_1,\ldots,h_N), \BS{x}:=(x_1,\ldots, x_N)
\]
for some
$N\ge1$, with the properties
\[
0<h_1<h_2<\cdots <h_N,  \qquad   x_1,
x_2,\ldots, x_N\in\R\setminus\{0\},
\]
and whenever $i<j$ and $x_i$, $x_j$ have the same sign, then
$|x_i|\le|x_j|$.

For notational convenience, we will also use the
indices $0, N+1, \infty$, and set
\[
x_0:=0, \qquad
x_\infty:=-x_1 \quad  \mbox{and} \quad   h_0:=0,  \qquad  h_{N+1}:=h_{\infty}:=2h_N.
\]
The mesh of the partition of $[0, h_N]$ induced by $\BS{h}$ is the number
\[
\mesh(\BS{h}):=\min\{h_i-h_{i-1}\dvtx  1\le i \le N\}.
\]

For a pair $(\BS{h}, \BS{x})$ as above, we let $\C{I}:=\{1,\ldots,
N\}$ and $\C{I}_\infty\subset\C{I}$ the largest set of
consecutive integers in $\C{I}$ containing $N$ and for which all
$x_i$ for $i\in\C{I}_\infty$ have the same sign.

For an index $i\in\C{I}$, let $i^-$ denote the greatest index
$j\in\C{I}$ less than $i$ so that $x_i$ and $x_j$ have the same sign, and
let $i^-=0$ if there is no such index. Similarly, let $i^+$
denote the least index $j\in\C{I}$ greater than $i$ so that $x_i$ and
$x_j$ have the same sign, and let $i^+=\infty$ if there is no such
index. In particular, $N^+=\infty$. Also let $\alpha, \beta$
denote the first index $i$ with positive and negative $x_i$,
respectively, again with the value $\infty$ if there is no such
index.

Consider the function $\Phi_{\mathbf{h}, \mathbf{x}}$ with domain
$[0, \infty)$
and value $x_0=0$ on the interval $[0, h_1]$, $x_i$ on the
interval $(h_i, h_{i+1}]$ for $i\in\{1, \ldots, N-1\}$ and $x_N$
on $(h_N, \infty)$; see Figure~\ref{stepFunction}. Recall the
definition of the graph occupation measure $m( \cdot )$ from the
\hyperref[sec1]{Introduction}, and let
\[
\mu_{\mathbf{h}, \mathbf{x}}:=m(\Phi_{\mathbf{h}, \mathbf{x}}).
\]
We will use the shorthand notation $I(\mathbf{h}, \mathbf{x})$ for
the rate
corresponding to this measure, namely
%
\begin{equation}\label{stepRate}I(\mathbf{h}, \mathbf{x}):=I(\mu
_{\mathbf{h}, \mathbf{x}})=\frac{\pi^2}{2}\sum_{i\in\C
{I}\setminus\C{I}_\infty}
\frac{|x_i-x_{i^-}|}{h_i^2}+\frac{\pi^2}{8}\sum_{i\in\C{I}_\infty
} \frac{|x_i-x_{i^-}|}{h_i^2}.
\end{equation}

\section{Confining Brownian motion---the upper bound}\label{sec3} \label
{UpperBoundSection}

The goal of this section is to prove the core of the large deviations
upper bound
of Section~\ref{LDPSection}.

In what follows, $m$ denotes the graph occupation
measure, $B$ a standard two-sided Brownian motion and $x$ the
process-of-wells mapping.

\begin{proposition}[(Large deviation upper bound)] \label{PropdensityBelow}
For each $\mu\in\C{M}$ and $A<I(\mu)$, there exists an open
neighborhood ${\mathcal U}$ of $\mu$ so that for all sufficiently
large $M$, we have
\[
\PP \bigl(m (x_B/M ) \in{\mathcal U}\bigr) \le e^{-AM}.
\]
\end{proposition}

The proof of this proposition is given in Lemmas
\ref{densityBelow} and~\ref{probabLemma}. We first define
neighborhoods that will be easy to handle. Using the notation of
Section~\ref{stepSection}, for $(\mathbf{h}, \mathbf{x})\in\C{S}$ and
$\eps>0$, we define the following open set of measures:
%
\begin{eqnarray}\label{neighborhood}
&&{\mathcal U}(\mathbf{h}, \mathbf{x}, \eps)\nonumber\hspace*{-35pt}
\\[-5pt]
\\[-10pt]&& \quad := \left\{\nu\in
\mathcal{M}\dvtx
\begin{array}{l@{ \qquad }l}
\nu \bigl((h_i-\eps, h_i+\eps)\times(x_i, \infty)  \bigr)>0 & \mbox
{for $i\in\C{I}$, } x_i>0 \\
\nu \bigl((h_i-\eps, h_i+\eps)\times(-\infty, x_i) \bigr) >0& \mbox
{for $i\in\C{I}$, } x_i<0
\nonumber
\end{array}
 \right\}.\hspace*{-35pt}
\end{eqnarray}
We claim that these neighborhoods cover everything efficiently,
even for small $\eps$.

\begin{lemma} \label{densityBelow}
For each $\mu\in\C{M}$ and $A<I(\mu)$, there exists
$(\mathbf{h}, \mathbf{x})\in\mathcal S$ so that $I( \mathbf{h},
\mathbf{x})>A$ and
${\mathcal U}(\mathbf{h}, \mathbf{x}, \eps)\ni\mu$ for all $\eps>0$.
\end{lemma}

The proof of this topological lemma is standard, but a bit
technical. We postpone it to Section~\ref{s.topology}.

\begin{lemma} \label{probabLemma} For $(\mathbf{h}, \mathbf{x})\in
\mathcal S$,
$A<I(\BS{h}, \BS{x})$ and all small enough $\eps>0$, there is an
integer $M_\eps$ so that \label{probabLemma2}
%
\begin{equation}\label{stepUpperBound}
\PP\bigl (m (x_B/M ) \in{\mathcal U}(\mathbf{h}, \mathbf{x},
\eps) \bigr)\le e^{-AM}  \qquad\mbox{for all }M\ge M_\eps.
\end{equation}
\end{lemma}

\begin{figure}[b]

\includegraphics{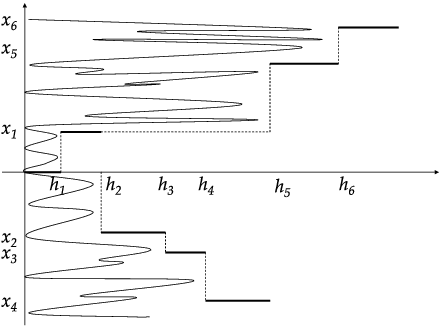}

\caption{The restrictions on the reflected process.} \label{reflected}
\end{figure}

For a two-sided Brownian motion $B$, we define its reflection from
its past minimum as the process
%
\begin{equation}\label{RBM}
R(t):=B(t)-\inf\{B(s)\dvtx  s \mbox{ between 0 and } t\}
\end{equation}
for all $t\in\D{R}$. This process appears naturally in the study
of the wells created by $B$.\vadjust{\goodbreak}

\begin{pf*}{Proof of Lemma~\ref{probabLemma}}
For a locally bounded function $Q\dvtx \D{R}\to\D{R}$, define
%
\begin{eqnarray}\label{runningMin}
\ul{Q}(t)&:=&\inf\{Q(s)\dvtx  s \mbox{ between 0 and } t\},
\\\label{runningMax}
\ol{Q}(t)&:=&\sup\{Q(s)\dvtx  s \mbox{ between 0 and } t\}
\end{eqnarray}
for all $t\in\D{R}$.

Recall the mapping $\mu\mapsto(f_\mu, g_\mu)$ defined in the
\hyperref[sec1]{Introduction}. For almost all Brownian paths $B$, the measure
$\mu:=m (x_B/M )$ satisfies
\[
Mf_\mu=f_{m(x_B)}=\overline{ x_B},\qquad M
g_\mu=g_{m(x_B)}=\overline{ (x_B)^-}.
\]
Also, $\mu\in{\mathcal U}(\mathbf{h}, \mathbf{x}, \eps)$
implies $f_\mu(h_i+\eps)>x_i$ for $x_i>0$, and similarly
for $x_i<0$. Then, for $i\in\C{I}$ with $x_i>0$, we have
\[
\overline{ x_B}(h_i+\eps)>Mx_i  \quad \Rightarrow \quad \ol{R}(M x_i)<h_i+\eps,
\]
because otherwise an ascent on the right with height at least
$h_i+\eps$ is created before $M x_i$. This can be paired
with an ascent on the negative axis of height at least $h_i+\eps$, and
the two will make $\ol{x_B}(h_i+\eps)$ to be located in $(-\infty,
M x_i]$, a~contradiction. This and the symmetric argument for
negative $x_i$ shows that, on the event in the statement of the
lemma, we have
\[ \label{upRestrictions}
\ol{R}(M x_i)<h_i+\eps \qquad \mbox{for all } i\in\C{I}.
\]
For $M=1$, a realization of the process $R$ satisfying these
restrictions is depicted in Figure~\ref{reflected}.

There is one more piece of information we have for the path at the
points $Mx_i$ for all indices $i\in\C{I}\setminus\C{I}_\infty$
when this
set is nonempty. That is,
\[
\underline{B}(M x_i)\ge-h_N-\eps.
\]
To see this, assume without loss of generality that $x_i>
0$. Since $\mu\in{\mathcal U}(\mathbf{h}, \mathbf{x}, \eps)$, we have
\[
\mu \bigl((h_i-\eps, h_i+\eps)\times(x_i, \infty) \bigr)>0,
\]
and thus $x_B(h')>M x_i>0$ for some $h'\in(h_i-\eps,
h_i+\eps)$. Let
\[
h^*:=\inf\{h>h'\dvtx x_B(h)<0\}.
\]
We first argue that $h^*$ is well defined, that is, the above
set is not empty. Since $i\in\C{I}\setminus\C{I}_\infty$,
there is $j>i$ with $x_j<0$, and thus $x_B(h)<0$ for some
$h\in(h_j-\eps, h_j+\eps)$. Since $\eps<
\operatorname{mesh}(\BS{h})/2$, we have $h_j-\eps>h_i+\eps$, and so
indeed we have $h>h'$. Also, $h^*<h_N+\eps$.

At $h^*$, $x_B$ is
positive because it is left continuous, but just after that
it is negative. This means that the well of $x_B(h^*)$
has depth exactly $h^*$. But
$B(x_B(h^*))=\ul{B}(x_B(h^*))$, and combining this with
$x_B(h^*)\ge x_B(h')>M x_i$, we obtain
\[
-h_N-\eps<-h^* \le\ul{B}(x_B(h^*)) \le\ul{B}(M x_i).
\]
Thus the event of the lemma is
contained on the event
\[
C_M:= \left\{
\begin{array}{l@{ \qquad }l}
\ol{R}(M x_i)<h_i+\eps&\mbox{for } i\in\C
{I},\\
\underline{B}(M x_i)\ge-h_N-\eps&\mbox{for } i\in\C
{I}\setminus\C{I}_\infty
\end{array}
 \right\}.
\]
Recall from Section~\ref{stepSection} that $i^-$ refers to the
index preceding $i$ so that $x_i$ and $x_{i^-}$ have the same
sign. Let $\PP_{r,y}$ denote the law of the Markov process
$(\overline{R},\underline{B})$ started at the point $(r,y)$. By
the Markov property applied consecutively at $M x_{i^-}$ for $i\in
\C{I}$, we get
\begin{eqnarray*}\PP(C_M)&\le& \prod_{i\in\C{I}\setminus\C{I}_\infty}
\sup_{r\ge0, y\le0} \PP_{r,y}\bigl(\ol{R}\bigl(M
(x_i-x_{i^-})\bigr)<h_i+\eps,\\
&&\hphantom{\prod_{i\in\C{I}\setminus\C{I}_\infty}
\sup_{r\ge0, y\le0} \PP_{r,y}\bigl(}
\underline{B}\bigl(M (x_i-x_{i^-}) \bigr)\ge-h_N-\eps\bigr) \\
&&{}\times \prod_{i\in\C{I}_\infty} \sup_{r\ge0, y\le0} \PP
_{r,y}\bigl(\ol{R}\bigl(M (x_i-x_{i^-})\bigr)<h_i+\eps\bigr).
\end{eqnarray*}
As usual, the product over an empty index set is 1.
Note that the process $(\overline{R},-\underline{B})$ is
nondecreasing in both coordinates of its starting point $(r,-y)$.
Therefore we have the upper bound
\begin{eqnarray*}
&&\prod_{i\in\C{I}\setminus\C{I}_\infty}\PP_{0,0}\bigl(\ol{R}\bigl(M
(x_i-x_{i^-})\bigr)<h_i, \underline{B}\bigl(M (x_i-x_{i^-}) \bigr)\ge
-h_N-\eps\bigr)
\\&& \qquad {}  \times\prod_{i\in\C{I}_\infty} \PP_{0,0}\bigl(\ol{R}\bigl(M
(x_i-x_{i^-})\bigr)<h_i+\eps\bigr).
\end{eqnarray*}
Then Lemma~\ref{confinementCost} implies that
\[
\lim_{M\to\infty}
\frac{\log\PP(C_M)}M\le-I(\BS{h}, \BS{x})-o(\eps),
\]
where
$o(\eps)$ depends on $(\mathbf{h}, \mathbf{x})$ only. The claim follows.
\end{pf*}

\section{Making a vessel---the lower bound}\label{sec4} \label{specialSets}

The goal of this section is to prove the large deviation
lower bound. This can be formulated as follows:

\begin{proposition}[(Large deviation lower bound)]
\label{PropLowerLDP} For every open set $G\subset\mathcal
M$, every $A>\inf_G I$ and for all sufficiently large $M$,
we have
\[
\PP\bigl(m(x_B/M)\in G\bigr) \ge e^{-AM}.
\]
\end{proposition}

Again, we proceed in two steps. We will first define a
convenient set of Brownian paths, $\mathcal R(\mathbf{h}, \mathbf
{x},\eps,
\delta)$, and then prove a topological lemma that reduces
the problem to showing that these paths have high
probability.

\begin{lemma} \label{rateFunctionApprox}
For every open $G\subset\mathcal M$, and every $A>\inf
_G I$, there exists $(\mathbf{h}, \mathbf{x})\in\mathcal S$ so that
$I(\mathbf{h}, \mathbf{x})<A$ and
\[
\{m(x_B)\dvtx  B\in\mathcal R(\mathbf{h}, \mathbf{x},\eps, \eps)\}
\subset G
\]
for all small enough $\eps>0$.
\end{lemma}

The proof of this
topological lemma is postponed to Section~\ref{s.topology}.
In light of this lemma, it suffices to give a lower bound
on the probability that for $B$ two-sided Brownian motion,
$m(x_B)$ is close to $\mhx$ in the large deviation regime.
In fact, we will do this for a more restrictive set, namely for the
event that $x_B$ is close in the Skorokhod topology to $\phx$.

Recall the Skorokhod topology on left continuous paths on $[0,\infty)$
with right limits. We call a set $A$ of these paths an
$[a,b]$-Skorokhod neighborhood of $f$ if it is the inverse image
of a Skorokhod neighborhood of $f|[a,b]$ under the restriction
map.

\begin{proposition}\label{vessel}
Let $(\BS{h}, \BS{x})\in\mathcal S$. Then every
$[0,2h_N]$-Skorokhod neighborhood of $\phx$ contains the image of the
set $\Rhxde$ under the map $x$ for all $ \delta, \varepsilon$ small
enough, and Brownian motion $B$ satisfies
%
\begin{equation} \label{RsetCost}
\lim_{M\to\infty} \frac{1}{M}\log\PP\bigl(B(M  \cdot  )\in\Rhxde\bigr)
=-I(\BS{h}, \BS{x})+O_{\BS{h}, \BS{x}}(\gd, \eps).
\end{equation}
\end{proposition}

The rest of this section contains the proof of the proposition.
First, we construct the desired set of paths, then we show that
they can be arbitrarily close to $\phx$, and finally we prove the
desired probability decay.

Each path in $\Rhxde$ forms a ``vessel,'' in which when
water is poured from the $y$ axis, the process of wells is
close to $\phx$ almost until depth $2h_N$ is reached.

\subsection*{Construction}

We define the following events, that is, sets of continuous
functions $f\dvtx \D{R}\to\D{R}$,
for $\eps, \gd\in(0, 1), h>0$ and $x, y\in\R$ with $0\le x<y$ or
$y<x\le0$.
For all events, we require that, in the second endpoint of the interval
mentioned, $f$ takes a value in $[0,h-\eps h]$; this is to make the
building blocks fit
together well. In addition, we require:\vadjust{\goodbreak}
\begin{enumerate}
\item[\textit{Confinement,
$C(x,y,h)$}:] between times $x(1+\delta),y(1-\delta)$, $f$ stays
in $[-\eps^2 h,h]$.
\item[\textit{Hole,
$H(y,h)$}:] between times $y(1-\delta), y$, $f$ stays in
$[-\eps h,h]$, visits below $-\eps h+\eps^2 h$.
\item[\textit{Hole$^R$,
$H^R(y,h)$}:] between times $y(1-\delta),y$, $f$ stays in
$[0,h]$, visits $0$.
\item[\textit{Barrier,
$B(y,h)$}:] between times $y,y(1+\delta)$, $f$ stays in
$[-\eps^2 h,h+\eps h]$, visits above $h$.
\end{enumerate}

We omit the dependence on $\eps, \gd$ from the notation.

Our basic restriction set, $\Rhxde$, is defined as the
intersection of the following sets $E_i$, for
$i\in\C{I}\cup\{0\}$. Our goal is to ensure that functions $f$ in
these sets will have the property that $x_f$ is
$[0,2h_N]$-Skorokhod-close to $\phx$. We will comment on the
importance of the individual sets $E_i$ after their definition.

\subsubsection*{The beginning}
%
\begin{eqnarray}
\label{start}
E_0&:=&C\bigl(0, x_\alpha\gd/(1-\gd), h_\alpha\bigr)\cap B(x_\alpha\gd,
h_\alpha)\nonumber
\\[-8pt]
\\[-8pt]
 &&{}\cap C\bigl(0, x_\beta\gd/(1-\gd), h_\beta\bigr)\cap
B(x_\beta\gd, h_\beta).
\nonumber
\end{eqnarray}
For $f\in E_0$, we have
\[
x_f(h)\in\bigl(\gd(1+\gd)x_\beta, \gd(1+\gd)x_\alpha\bigr)    \qquad \mbox{for }  h\in[0, h_1],
\]
because the two barrier sets create a well around
zero of depth at least $h_1$.

\subsubsection*{The indices in $\C{I}\setminus\C{I}_\infty$}
For each index $i\in\C{I}\setminus\C{I}_\infty$, we
define the set
\[
E_i:=C(w_i, x_i, h_i)\cap H(x_i, h_i)\cap B(x_i, h_{i^+}),
\]
where for all $i\in\C{I}$, we let
\[
w_i:=
\cases{\displaystyle x_{i^-}, &\quad $i\ne\alpha, \beta$,\cr\displaystyle
x_i\delta,&\quad $i=\alpha$  or   $\beta$.
}
\]
The purpose of $E_i$ is to guarantee that for $f\in E_0 \cap\cdots \cap E_i$, the value of $x_f( \cdot )$ will be near
$x_i$ for a time interval very close to $[h_i, h_{i+1}]$.

\begin{figure}

\includegraphics{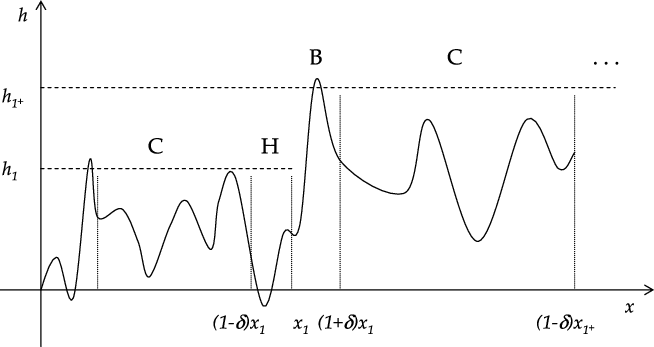}

\caption{The first positive blocks of the construction.} \label{FirstBlocks}
\end{figure}

More precisely, assume that $x_i>0$, and focus on the positive
half of the path $f|[0, \infty)$. See Figure~\ref{FirstBlocks} for
the case $i=1$. What $E_i$ adds to the
intersection is that $f|[0, \infty)$ up to the point $x_i(1-\gd)$
does not reach a new minimum or maximum. Then it creates a new
minimum (hole) near $x_i$, and then a barrier of height $h_{i^+}$
ahead of it.
But $E_{i^-}$ has already created a barrier of height about
$h_i$.
Moreover, on the negative side there is a barrier of height at
least $h_{i+1}$ following a minimum for $f|(-\infty, 0]$ which is
not deeper than the hole in the event $E_i$. So indeed, the value
of $x_f( \cdot )$ will be near $x_i$ at least for a time interval almost
equal to $[h_i, h_{i+1}]$.

Note also that the barrier created by $E_i$ makes sure that from height
about $h_i$ until
height about $h_{i^+}$, the process $x_f$ either stays constant or
jumps to negative values; that is, it does not advance to another
positive value.

\subsubsection*{The indices in $\C{I}_\infty$}
By symmetry, we may assume that the $x_i$'s for $i\in
\C{I}_\infty$ are positive. For a locally bounded function $f$
defined in $\D{R}$, and $z$ fixed, let
$R_zf\dvtx [z,\infty)\to[0,\infty)$ denote $f$ reflected from its
running minimum after $z$, namely
%
\begin{equation}\label{reflectionFromInf}
R_zf(x):=f(x)-\inf_{s\in[z, x]} f(s).
\end{equation}
Let $q=\min\C{I}_\infty$. For $i\in\C{I}_\infty$, define the
set $E_i$ of paths $f$ so that $R_{w_q(1+\gd)}f$ is in
\[
C(w_i, x_i, h_i)\cap H^R(x_i, h_i)\cap B(x_i, h_{i^+}),
\]
and $f$ satisfies
%
\begin{equation}\label{res2}
f(x)- f\bigl(w_i(1+\gd)\bigr) \le\eps^2
 \qquad \mbox{for } x\in[w_i(1+\gd), x_i(1+\gd)].
\end{equation}
Note that $i^+=i+1$ unless $i=N$.

In order to understand these events $E_i$, we first
consider the effect of the preceding events $E_0\cap\cdots \cap E_{q-1}$. See Figure~\ref{lastpic}.

Assume first that $\C{I}_\infty\not=\C{I}$. In this case, $E_0\cap\cdots \cap E_{q-1}$ puts restriction on the path on the interval
$[x_{q-1}(1+\gd), w_q(1+\gd)]$. The minimum value of $f$ there is
negative of order $\eps$, the maximum is attained in the interval
$[x_{q-1}(1+\gd), x_{q-1}]$, where the path goes over $2h_N$ because
$(q-1)^+=\infty$ and $h_\infty=2h_N$.

When $\C{I}_\infty=\C{I}$, the event $E_0$ puts restriction on
the path on the interval $[-x_1\gd(1+\gd), w_q(1+\gd)]$. The
minimum value of $f$ there is negative of order~$\eps$, the
maximum is attained on $[-x_1\gd(1+\gd), -x_1\gd]$, where the path
goes over~$2h_N$.

In both cases, the maximum on $[0, w_q(1+\gd)]$ is attained in the
interval $[w_q, w_q(1+\gd)]$, where the path goes a bit over $h_q$
and ends up in $[0, h_q-\eps h_q]$.\vadjust{\goodbreak}

\begin{figure}

\includegraphics{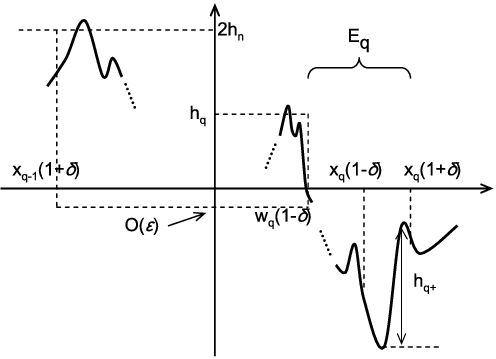}

\caption{The last part.} \label{lastpic}
\end{figure}

Then the set $E_q$ requires from $f$ up to the point
$x_q(1-\gd)$ not to create an ascent of height larger than
$h_q$ (see Figure~\ref{lastpic}) and then to create one
with lowest point having $x$-coordinate in $(x_q(1-\gd),
x_q(1+\gd))$ and height around $h_{q^+}$. The goal of
\eqref{res2} is to force $f$ not to go above $h_q$, and
this is obtained because as we noted $f(w_q(1+\delta))\le
h_q-\eps h_q$, and from that point on $f$ stays below
$h_q-\eps h_q+\eps^2$. These, together with the barrier of
height $2h_N$ on the negative axis, guarantee that $x_f(h)$
is around $x_q$ at least for $h$ in an interval very close
to $[h_q, h_{q^+}]$.

The other $E_i$'s with $i\in\C{I}_\infty$ work in the same way.

\subsection*{The behavior of $x_f$ for $f\in\Rhxde$}
We will now examine more precisely how $x_f$ behaves when $f\in
\mathcal R(\BS{h}, \BS{x},\gd, \eps)$. For $x,y\in\mathbb R$,
define
\[
f^\#(x,y)=
\cases{\displaystyle\sup\{f(t)-f(s)\dvtx x\le s\le t\le y\} ,&\quad $x \le
y$,\cr\displaystyle
\sup\{f(t)-f(s)\dvtx y\le t\le s\le x\} ,&\quad $x \ge y$.
}
\]
Also let
\[
\tilde
h_1:=\max\{f(x)\dvtx  x \mbox{ between 0 and } x_1 \gd(1+\gd)\},
\]
and
let $z_1$ be the closest to zero point between $0$ and
$x_{\alpha\vee\beta} \delta(1+\gd)$ where $f$ takes the value
$\tilde h_1$. In the interval between $z_1$ and $x_1\gd(1+\gd)$,
we have a well of depth around $h_1$. Its exact depth is
\[
v_1:=\tilde h_1-\min\{f(x)\dvtx x\mbox{ between $z_1$ and } x_1\gd(1+\gd
)\}.
\]
Also define
\[
v_i:=
\cases{\displaystyle
f^\#\bigl(x_{i-1}, x_i\gd(1+\gd)\bigr) ,&\quad   if   $i=\alpha\vee\beta\ne
\infty $,\cr\displaystyle
f^\#\bigl(x_{i-1}(1-\gd), x_{i-1}(1+\gd)\bigr) ,&\quad   if   $i-1= i^-\in\C
{I}\setminus\C{I}_\infty $,\cr\displaystyle
f^\#\bigl(x_{i-1}, x_{i^-}(1+\gd)\bigr) ,&\quad   if   $i-1\ne i^-\in\C
{I}\setminus\C{I}_\infty$.\vadjust{\goodbreak}
}
\]
Now for $i^-\in\C{I}_\infty$ let
\[
v_i:=\sup\bigl\{R_{w_q(1+\gd)}f(x)\dvtx  x\in[x_{i^-}, x_{i^-}(1+\gd)]\bigr\},
\]
and finally let $v_{N+1}= 2h_N$. From the above discussion and the
definition of $\mathcal R(\BS{h}, \BS{x},\gd, \eps)$, we conclude
that
%
\begin{equation}\label{hvariation}
v_i\in
\cases{\displaystyle
[h_1, h_1+\eps(h_1+\eps h_{\alpha\vee\beta})] ,&\quad   if   $i=1 $,\cr\displaystyle
[h_i+(\eps-\eps^2)h_{i-1}, h_i+\eps(h_i+h_{i-1})]
,&\quad   if   $i=\alpha\vee\beta\ne\infty$\cr
& \qquad   or   $i^-\in\C
{I}\setminus\C{I}_\infty $,\cr\displaystyle
(h_i, h_i+\eps h_i] ,&\quad   if   $i^-\in\C{I}_\infty$,
}
\end{equation}
and
%
\begin{eqnarray}\label{xvariation}
&\displaystyle |x_f(h)|\le\gd(\gd+1) (x_\alpha\vee|x_\beta|)  \qquad \mbox{for } h\in
[0, v_1],& \nonumber
\\[-8pt]
\\[-8pt]
&\displaystyle x_f(h)\mbox{ is between }x_i(1-\gd), x_i  \qquad \mbox{for }
h\in(v_i, v_{i+1}],    i\in\C{I}.&
\nonumber
\end{eqnarray}
We assumed that $\eps$ is small enough so that $v_1<v_2<\cdots <v_N$,
and $-\eps h_i+\eps^2h_i<-\eps h_{i-1}$ for $i\in\C{I}\setminus\C
{I}_\infty$. Informally, the second requirement guarantees that, for
these $i$'s, the set $H(x_i, h_i)$ creates a new, deeper minimum, and
this is used for \eqref{xvariation}.

Relations \eqref{hvariation}, \eqref{xvariation} show that
any $[0,2 h_N]$-Skorokhod neighborhood of
$\Phi_{\mathbf{h}, \mathbf{x}}$ contains $\{x_f\dvtx f\in\mathcal
R(\mathbf{h}, \mathbf{x}, \gd, \eps)\}$ if $\eps$, $\delta$ are
small enough.

\subsection*{The asymptotic probability of $\Rhxde$}
It remains to prove \eqref{RsetCost}.

We apply the Markov property and use Lemma~\ref{blocksCosts}. Note
that for any two restriction sets concerning contiguous intervals,
say $[x, y], [y, z]$, with $0<x<y<z$, the allowed values for
$f(y-)$ are the same as the ones on which we condition in the
first three relations of Lemma~\ref{blocksCosts}. If the second
set corresponds to an index in $\C{I}_\infty$, the ending point
of the first block is irrelevant.

It is also important that the limits computed in that lemma are uniform
over the starting points of the processes involved. These observations
allow us to conclude that the left-hand side of \eqref{RsetCost} equals
\begin{eqnarray*}
&&-\frac{\pi^2}{2}\sum_{i=\alpha, \beta}\frac{|x_i|\gd}{h_i^2}
 \biggl(\frac{1}{(1+\eps^2)^2}+ \frac{\gd}{(1+\eps+\eps
^2)^2} \biggr)
\\
&& \qquad {}-\frac{\pi^2}{2}\sum_{i\in\C{I}\setminus\C{I}_\infty}
\biggl(\frac{|x_i-w_i-\gd(x_i+w_i)|}{h_i^2(1+\eps^2)^2}+\frac{\gd
|x_i|}{h_i^2(1+\eps)^2}+
\frac{\gd|x_i|}{h_{i^+}^2(1+\eps+\eps^2)^2}  \biggr) \\
&& \qquad {}-\frac{\pi^2}{8}\sum_{i\in\C{I}_\infty}  \biggl(\frac
{|x_i-w_i-\gd w_i|}{h_i^2}+ \frac{\gd|x_i|}{h_{i^+}^2(1+\eps
)^2} \biggr),
\end{eqnarray*}
which is $-I(\BS{h}, \BS{x})+O_{\BS{h}, \BS{x}}(\gd,
\eps)$. Note that for $\gd\searrow0$, only the confinement
sets appearing in $E_i$, for $i\in\C{I}$, contribute to
the rate of decay. The reason is that all other sets put
restrictions on intervals of size proportional to $\gd$.
Similarly, the first restriction set $E_0$ does not
contribute.

\section{The large deviation principle for the process of wells}\label{sec5}\label
{LDPSection}
In this section, we complete the proof of the large
deviation principle for the family $\{m (x_B/M )\dvtx
M>0\}$ stated in Theorem~\ref{bLDP}.

Recall that a family $\{\mu_M\dvtx  M>0\}$ of Borel measures on
a topological space $\mathcal M$ satisfies the \textit{large
deviation principle} with rate $I\dvtx \mathcal M \to
[0,\infty]$ as $M\to\infty$ if for every measurable set
$\mathcal B\subset\mathcal M$ and every $A>\inf_{\mathcal
B^{\circ}} I$ and $A'<\inf_{\bar{\mathcal B}} I$ we have
%
\begin{equation}\label{LDPrecall}
-A\le\frac{\log\mu_M(\mathcal B)}{M} \le-A'
\end{equation}
for all sufficiently large $M$, where ${\mathcal B}^\circ,
\bar{\mathcal B}$ denote the interior and the closure of
$\mathcal B$, respectively.

Recall also that $I$ is a \textit{good rate function} if
$I^{-1}[0,A]$ is compact for all finite $A$. In
particular, these sets are closed, which is equivalent to
$I$ being lower semicontinuous.

We have established in Propositions~\ref{PropdensityBelow}, \ref
{PropLowerLDP} the core upper and lower bounds. Next,
we prove exponential tightness. Recall that a family of
measures $\{\mu_M\dvtx  M>0\}$ as above is \textit{exponentially tight} as
$M\to\infty$ if
for every $A>0$ there is a compact set $\mathcal Q\subset\mathcal{M}$
so that
$\mu_M(\mathcal Q^c)<e^{-AM}$ for all large enough~$M$.

\begin{lemma}\label{exptightness}
The family $\{m (x_B/M )\dvtx
M>0\}$ is exponentially tight.
\end{lemma}

\begin{pf}
By the definition of $\mathcal M$, for every $a>0$, the set
%
\begin{equation}
\label{Qeq}
\mathcal Q_a:= \Biggl\{\mu\in\C{M}\dvtx  \supp(\mu)\subset
\bigcup_{k=1}^\infty[k-1, k]\times[-ak^3, ak^3] \Biggr\}
\end{equation}
is compact in $\C{M}$.
Recall definitions \eqref{RBM} and \eqref{runningMax}. We
have
%
\begin{eqnarray}
\label{tightnessSet} \PP\bigl(m (x_B/M)\in
\mathcal Q_a^c \bigr)&\le&\sum_{k=1}^\infty\PP\bigl(\max
\{\ol{x_B}(k), \ol{(x_B)^-}(k)\}\ge aM k^3\bigr)\nonumber
\\[-8pt]
\\[-8pt]&\le& 2
\sum_{k=1}^\infty\PP\bigl(\ol{x_B}(1)\ge aM k\bigr)
\nonumber
\end{eqnarray}
using the scaling and symmetry properties of $x_B$. But for
all $x>0$, we have that $\ol{x_B}(1)\ge x$ implies $\bar R(x)<1$,
which has probability at most
\[
\PP\bigl(B[0, x]\subset(-1,
1)\bigr)\le C e^{-x\pi^2/8}
\]
with $C$ a constant. Here we used the fact that $R$ has
the same law as $|B|$ and relation \eqref{PortStoneIneqU}. Consequently,
%
\begin{equation}\label{Qconclusion}
\PP\bigl(m (x_B/M)\in\mathcal Q_a^c \bigr)\le
C'e^{-Ma\pi^2/8}
\end{equation}
for a constant $C'$. Since $a$ was arbitrary, exponential tightness follows.\vadjust{\goodbreak}
\end{pf}

\begin{pf*}{Proof of Theorem~\ref{bLDP}, the large deviation principle}
The first inequality in \eqref{LDPrecall} is a
reformulation of Proposition~\ref{PropLowerLDP} applied to
the open set $\mathcal B^\circ$. For the second, let
$A<\inf_{\bar{\mathcal B}} I$. By exponential tightness
(Lemma~\ref{exptightness}) there exists a compact set
$\mathcal Q$ so that
\[
P\bigl(m(x_B/M)\in\mathcal Q^c\bigr) \le e^{-AM}
\]
for all sufficiently large $M$. Each point in $\mathcal
Q\cap\bar{\mathcal B}$ can be covered with an open set
satisfying the same asymptotic bound by Proposition
\ref{PropdensityBelow}. To get the second inequality of
\eqref{LDPrecall}, take a finite subcover of the compact
set $\mathcal Q\cap\bar{\mathcal B}$, and use the union
bound.

Finally, we show that $I$ is a good rate function.
Take $A>0$ and $\mu\in I^{-1}(A,\infty]$. By Proposition \ref
{PropdensityBelow},
$\mu$ has an open neighborhood $\mathcal U$ so that
\[
-A > \limsup_{M\to\infty} \frac{\log P(m(x_B/M)\in\mathcal
U)}{M}.
\]
The right-hand side is bounded below by $-\inf_{\mathcal U}
I$ because of the large deviation lower bound. Thus $\mathcal
U\subset I^{-1}(A,\infty]$, which shows that the latter set must be open.
Thus $I^{-1}[0,A]$ is closed. On the other hand,
exponential tightness and the large deviation lower bound
gives $\inf_{\mathcal Q^c} I
>A$ for some compact $\mathcal Q$, so $I^{-1}[0,A]\subset\mathcal Q$
must also be compact.
\end{pf*}

\section{The limit points of the environment}\label{sec6} \label{EnvLIL}

For $a>1$, and $B\dvtx \D{R}\to\D{R}$ a continuous path, we define the
function $Z_a\dvtx [0, \infty)\to\D{R}$ by
%
\begin{equation} \label{zPath}
Z_a(s):= \frac{x_B(sa)}{a^2\log\log a}
\end{equation}
for all $s\ge0$. We will determine the limit points of the
family of measures $(m(Z_a))_{a>e}$, as $a\to\infty$, with respect to
the topology
of local weak convergence, when $B$ is a two-sided Brownian path. Then
Theorem~\ref{FunLIL} will follow from localization results connecting
$x_B$ with the Sinai walk. We
start by doing this along geometric sequences.

\subsection*{Geometric sequences}
In the next proposition, we show that all limit points of
$(m(Z_a))_{a>e}$ along geometric sequences fall into a
certain set $\C{K}$. Then Proposition~\ref{geomSequences2}
shows that in fact along any geometric sequence, all points
of $\C{K}$ are limit points.

\begin{proposition} \label{geomSequences1}
If $c>1$, then with probability one, every subsequence of
$\{m(Z_{c^n})\dvtx  n> 1/\log c \}$ has a further convergent subsequence, and
the limit points of the original sequence are contained in the set
\[
\C{K}:=\{\mu\in\C{M}\dvtx  I(\mu)\le1\}.
\]
\end{proposition}

\begin{pf}
By the scaling property of $x_B$ and \eqref{Qconclusion},
\[
\PP\bigl(m(Z_{c^n}) \in\mathcal Q_a^c \bigr)=\PP\bigl(m(x_B/\log\log c^n)\in
\mathcal Q_a^c\bigr)\le
C' \exp \biggl(- \frac{a\pi^2}8 \log\log c^n \biggr).\vadjust{\goodbreak}
\]
For $a>8/\pi^2$, the first Borel--Cantelli lemma implies that
$m(Z_{c^n}) \in\mathcal Q_a$ eventually, and the first claim follows
by the compactness of $\mathcal Q_a$.

For each point in $\mathcal K^c$, Proposition
\ref{PropdensityBelow} provides an open set $\mathcal U$
containing it so that for some $A(\mathcal U)>1$ and for
all large enough $n$, we have
%
\begin{equation}\label{AU}
\PP\bigl(m(Z_{c^n})\in\mathcal{U}\bigr)\le\exp\bigl(-A(\mathcal U)
\log\log c^n\bigr).
\end{equation}
Now each such set $\mathcal U$ can
be written as a union of elements of a fixed countable
base. Thus $K^c$ can be covered with a countable collection
of open sets $\mathcal U_k$ satisfying \eqref{AU}. By the
first Borel--Cantelli Lemma and the union bound, no
$\mathcal U_k$ contains a limit point a.s.
\end{pf}

The promised complement of Proposition~\ref{geomSequences1}
is as follows.

\begin{proposition} \label{geomSequences2}
If $c>1$, then with probability one, the limit points of
$\{m(Z_{c^n})\dvtx  n>1/\log c\}$ include the points of the set
\[
\C{K}:=\{\mu\in\C{M}\dvtx  I(\mu)\le1\}.
\]
\end{proposition}

\begin{pf}
Note that it suffices to prove that every open set
$\mathcal U$ intersecting $\mathcal K$ contains a limit
point with probability one. Using this claim for all such
elements $\mathcal U$ of a countable base for $\mathcal M$,
we conclude that the limit points are a.s. dense in
$\mathcal K$. Since they form a closed set, this set must
contain $\mathcal K$.

By Lemma~\ref{l:nomin} the minimum of $I$ on an open set is
either $0$, $\infty$ or is not achieved. Therefore every
open set $\mathcal U$ intersecting $\mathcal K$ has
$\inf_{\mathcal U} I <1$. By Lemma
\ref{rateFunctionApprox}, there exists $(\mathbf{h}, \mathbf{x})\in
\mathcal S$ so that $I(\mathbf{h}, \mathbf{x})<1$ and
\[
\{m(x_B)\dvtx  B\in\mathcal R(\mathbf{h}, \mathbf{x},\eps, \eps)\}
\subset
\mathcal U
\]
for all small enough $\eps>0$.

Define $n_0=\lfloor1/\log c \rfloor+2$, and for $n\ge
n_0$, let $A_n$ be the set of paths $B$ so that the
rescaling satisfies
\[
\frac{B( c^{2n}\log\log c^n\times\cdot)}{c^n} \in\mathcal R(\BS
{h}, \BS{x}, \eps,
\eps).
\]
Note that if a path $B$ belongs to $A_n$, then the corresponding path
$Z_{c^n}$ from \eqref{zPath} satisfies $m(Z_{c^n}) \in\mathcal U$.
Thus it
suffices to show that $A_n$ i.o. a.s.

Since
%
\begin{equation}\label{AnDefinition}
A_n=\mathcal R(\BS{h}  c^n, \BS{x}  c^{2n} \log\log c^n , \eps,
\eps),
\end{equation}
the scaling property of Brownian motion implies
\[
\PP(A_n)=\PP(\mathcal R(\BS{h}, \BS{x}   \log\log c^n, \eps,
\eps)).
\]
Then Proposition~\ref{vessel} gives
\[
\liminf_{n\to\infty} \frac{\log\PP(A_n)}{\log\log c^n} \ge
-I(\BS{h}, \BS{x})-O_{\BS{h}, \BS{x}}(\eps)>-(1-\gd_1)
\]
for some $\delta_1\in(0, 1-I(\BS{h}, \BS{x}))$ and all small
enough $\eps$. Consequently, there is an $n_1$ so that for $n\ge
n_1$ we have $\PP(A_n)\ge n^{-1+\gd_1}$. Then for all $n$,
%
\begin{equation}\label{plbound}
\sum_{k=n_0}^n \PP(A_k)>C n^{\gd_1}
\end{equation}
for an appropriate constant $C>0$. In particular, $\sum
\PP(A_k)=\infty$. In order to conclude that $A_n$ i.o., we
use a correlation bound given in the upcoming Lemma
\ref{secondMoment}. Let $\Sigma_n:=1_{A_{n_0}}+\cdots +
1_{A_n}$. We write
\[
\ev(\Sigma_n^2)=\ev\Sigma_n+ 2\mathop{\mathop{\sum}_{k=n_0 }}_{
l>k}^n\PP(A_k\cap A_l).
\]
Let $\Delta, C_0$ be as in Lemma~\ref{secondMoment}, and
$d_k:=\Delta+(\log\log k) /(2\log c)$. We bound the
probabilities $\PP(A_k\cap A_l)$, $n_0\le k<l\le n$, in one
of two ways, according whether $|k-\ell|\le d_k$, thus
getting for their sum the upper bound
\[
C_0\mathop{\mathop{\sum}_{k=n_0 }}_{
l-k > d_k}^n \PP(A_k)\PP(A_l)+ d_n \sum_{k=n_0}^n
\PP(A_k)
\le \biggl(\frac{C_0}{2}+\frac{d_n}{\ev(\Sigma_n)} \biggr)(\ev
\Sigma_n)^2.
\]
But $d_n/\ev(\Sigma_n)\to0 $ by \eqref{plbound}, so that
the Kochen--Stone lemma [\citet{DU},  Exercise 2.3.20] gives
%
\begin{equation}\label{KStone}
\PP(A_n \mbox{ i.o.}) \ge\limsup_{n\to\infty}
\frac{(\ev\Sigma_n )^2}{\ev(\Sigma_n^2)}\ge1/C_0.
\end{equation}
To prove that $\{A_n \mbox{ i.o.}\}$ holds a.s., we will prove that it
is a tail event, that is, that it belongs to the $\sigma$-algebra
$\bigcap_{t>0}\sigma (B(s)\dvtx |s|\ge t )$, and we will apply
Theorem 8.2.7 from \citet{DU}.

To see this, fix $t_0>0$. A function $f\dvtx \D{R}\to\D{R}$ belongs to
$A_n$ if its values on a certain interval around zero satisfy certain
conditions imposed by the sets whose intersection defines $A_n$. For
$n$ that satisfies $t_0<\eps(x_\alpha\wedge|x_\beta|)   c^{2n}\log
\log c^n$, we isolate the conditions concerning the values of $f$ on
$[-t_0, t_0]$ and write
\begin{eqnarray*}A_n&=&\mathcal R(\BS{h}  c^n, \BS{x}  c^{2n} \log\log
c^n , \eps, \eps)\\&=&C\bigl(0, c^{2n} (\log\log c^n)  x_\alpha\eps
/(1-\eps), c^n h_\alpha\bigr)\\
&&{}\cap
C\bigl(0, c^{2n} (\log\log c^n)  x_\beta\eps/(1-\eps), c^n h_\beta
\bigr)\cap C_n\\
&=&C\bigl(0, t_0/(1-\eps), c^n h_\alpha\bigr)\cap
C\bigl(t_0/(1+\eps), c^{2n} (\log\log c^n) x_\alpha\eps/(1-\eps), c^n
h_\alpha\bigr)\\&&{} \cap
C\bigl(0, -t_0/(1-\eps), c^n h_\beta\bigr)\\
&&{}\cap
C\bigl(-t_0/(1+\eps), c^{2n} (\log\log c^n) x_\beta\eps/(1-\eps), c^n
h_\beta\bigr)\cap C_n
\\&=&C\bigl(0, t_0/(1-\eps), c^n h_\alpha\bigr) \cap C\bigl(0, -t_0/(1-\eps),c^n
h_\beta\bigr) \cap A_n',
\end{eqnarray*}
where $C_n$ is the intersection of the remaining sets involved in the
definition of $A_n$, and
\begin{eqnarray*}A_n'&:=&C\bigl(t_0/(1+\eps), c^{2n} (\log\log c^n) x_\alpha
\eps/(1-\eps), c^n h_\alpha\bigr)\\ &&{}\cap
C\bigl(-t_0/(1+\eps), c^{2n} (\log\log c^n) x_\beta\eps/(1-\eps), c^n
h_\beta\bigr)\cap C_n.
\end{eqnarray*}
%
Now $\{ A_n \mbox{ i.o.}\} \subset\{ A_n' \mbox{ i.o.}\}$, but also
$\{ A_n' \mbox{ i.o.}\} \subset\{ A_n \mbox{ i.o.}\}$ since every
function $f$ in the first set belongs to
$C(0, t_0/(1-\eps), c^n h_\alpha) \cap C(0, -t_0/(1-\eps),c^n
h_\beta)$ provided that $\max\{|f(s)|\dvtx  |s|\le t_0\}<\eps^2(h_\alpha
\vee h_\beta)  c^{n}$. And the last inequality holds for all large
$n$ because $c>1$, and $f$ is bounded on $[-t_0, t_0]$, being continuous.
Since for all large $n$ we have $A_n'\in\sigma (B(s)\dvtx  |s|\ge
t_0 )$, it follows that $\{ A_n \mbox{ i.o.}\}\in\sigma
(B(s)\dvtx  |s|\ge t_0 )$, and this proves our assertion.
\end{pf}

The next lemma shows a version of near independence for the family
of sets $\{A_n\dvtx  n\ge1\}$, defined in the proof of Proposition
\ref{geomSequences2}, and uses the notation set up in that proof.

\begin{lemma}\label{secondMoment}
There are $\Delta, C_0\in(0, \infty)$ depending on $\BS{h},
\BS{x}, \eps$ such that
\[
\PP(A_k\cap A_l)\le C_0   \PP(A_k)\PP(A_l)
\]
for $k\ge n_0$ and
%
\begin{equation}\label{lkdistance} l-k>\Delta+\frac{1}{2}\frac{\log
\log k}{\log c}.
\end{equation}
\end{lemma}

\begin{pf}
For the pair $(\BS{h}, \BS{x})\in\C{S}$, we will use the notation
of Section~\ref{stepSection}.

We assume that $x_N>0$. Let $p$ equal $\max(\C{I}\setminus
\C{I}_\infty)$ if the set is nonempty, and $\infty$
otherwise. For any integer $n\ge n_0$, define
\[
J_n:=c^{2n}\log\log c^n [x_p(1+\eps), x_N(1+\eps)].
\]
Recall that $x_\infty=-x_1$. $J_n$ is the interval where
$A_n$ imposes restrictions on~$B$.

$A_l$ is the intersection of several requirements the first
of which [the two confinement sets of $E_0$ in \eqref{start}]
refers to the time interval
\[
F_l:=c^{2l}\log\log c^l [\eps x_\beta, \eps x_\alpha].
\]
We would like to have $l$ so large that $J_k$ will be in
the interior of $F_l$, so that knowing that $A_k$ happened
does not influence much the probability of $A_l$. We ensure
that $J_k\subset F_l/2$ by assuming that
%
\begin{equation}\label{rest}
l>k+\frac{1}{2\log c} \biggl\{\log \biggl(
\frac{|x_p|}{|x_\beta|}\vee\frac{x_N}{x_\alpha} \biggr) +\log
\frac{2(1+\eps)}{\eps} \biggr\}
\end{equation}
for the rest of the proof. Note that \eqref{rest} is
implied by \eqref{lkdistance} with an appropriate choice of
$\Delta$. We let
\[
\Apl=\{B|[0, \infty)\dvtx  B\in A\},\qquad A^-=\{B|(-\infty, 0]\dvtx
B\in A\}.
\]
It is enough to prove the claim of the lemma for the pairs
$\{\Apl_k, \Apl_l\}, \{A_k^-, A_l^-\}$ as they are independent. We
will do it for the first. Let
\[
\Theta_l:= h_\alpha c^l [-\eps^2 , 1]\supset h_\alpha c^l\biggl[-\frac
{\eps^2}{2},
1-\frac{\eps^2}{2}\biggr]=:\tilde\Theta_l
\]
and
\[
\Apl_{k,
l}:=\Apl_k \cap \{B(s) \in\Theta_l \mbox{ for }0<s\in
J_k \}.
\]
Paths $B$ in $A_l$ satisfy $B(F_l\cap[0, \infty))\subset\Theta_l$.
So that $\Apl_k\cap\Apl_l=\Apl_{k,
l}\cap\Apl_l$
since $J_k\subset F_l$. Let $\Apl_l(J_k)$ denote the paths that
satisfy the restrictions put by
$\Apl_l$ for the time interval $J_k$, and define $\Apl_l(J_l\setminus J_k)$
analogously.
Denote by $j_k$ the right endpoint of $J_k$, and let
\[
q(x):=\PP\bigl (\Apl_l(J_l\setminus J_k) | B_{j_k}=x \bigr).
\]
We have
\[
\Apl_k\cap\Apl_l=\Apl_{k, l}\cap\Apl_l\subset
\Apl_{k, l}\cap\Apl_l(J_l\setminus J_k),
\]
and the probability of
the right-hand side can be written as
\[
\ev [\one_{\Apl_{k,l}} q(B_{j_k}) ]\le\PP(\Apl_k) \max
_{\Theta_l} q.
\]
On the other hand,
\begin{eqnarray*}
\PP(\Apl_l) &=&
\PP\bigl(\Apl_l(J_l\setminus J_k)\cap\Apl_l(J_k)\bigr)=
\ev\bigl [\one_{\Apl_l(J_k)} q(B_{j_k}) \bigr]\\&\ge&\PP\bigl(\Apl
_l(J_k) \mbox{ and } B_{j_k}\in\tilde\Theta_l\bigr)  \min_{\tilde
\Theta_l} q.
\end{eqnarray*}
So that
%
\begin{equation}\label{CorrIneq}
\frac{\PP(\Apl_k\cap\Apl_l)}{\PP(\Apl_k)\PP(\Apl_l)}\le
\frac{\max_{\Theta_l} q}{\min_{\tilde\Theta_l} q}
\PP\bigl(\Apl_l(J_k)
\mbox{ and } B_{j_k}\in\tilde\Theta_l\bigr)^{-1}.
\end{equation}
To bound the last term, note that $B|[0, \infty) \in\Apl_l(J_k)$
follows from
$B([0,j_k]) \subset\tilde\Theta_\ell$. The restriction
\eqref{lkdistance} on $l-k$ shows that
\[
c^{2k-2l} \log\log c^k<c^{-2\Delta} \biggl(1+\frac{\log\log c}{\log
2} \biggr)=:c_1.
\]
This and Brownian scaling yield the lower bound
\[
\PP\bigl(\Apl_l(J_k) \mbox{ and } B_{j_k}\in\tilde\Theta_l\bigr)\ge\PP
 \biggl(B\bigl([0, c_1(1+\eps)x_N]\bigr)\subset h_\alpha\biggl[-\frac{\eps^2}{2},
1-\frac{\eps^2}{2}\biggr] \biggr),
\]
which is positive and does not depend on $k, l$.\vadjust{\goodbreak}

To bound the fraction in the right-hand side of \eqref{CorrIneq}, note
that with $f_l$ the right endpoint of $F_l$, the event
$\Apl_l(F_l\setminus J_k)$ is equivalent to $B([j_k,f_l]) \subset
\Theta_\ell$ and $B_{f_l}\in[0, (1-\eps) h_\alpha c^l]$, by the
definition of the confinement set $C(0, x_\alpha\eps/\break(1-\eps),
h_\alpha)$. Let $r(x,y)$ denote the density of $B(f_l)$ for
Brownian motion started from $x$ at time $j_k$ restricted to this
event. By the Markov property, we have
\[
q(x) = \int r(x,y) \pr \bigl(\Apl_l(J_l\setminus F_l) |
B(f_l)=y \bigr)\,dy,
\]
which gives the bound
\[
\frac{\max_{\Theta_l} q}{\min_{\tilde\Theta_l} q} \le\max
\biggl\{\frac{r(x_1,y)}{r(x_2,y)}\dvtx  x_1\in\Theta_l, x_2\in\tilde\Theta
_l, y\in[0, (1-\eps) h_\alpha c^l] \biggr\}.
\]
With the notation introduced in the beginning of Section \ref
{BrownianComputations}, we have
\[
r(x, y)=Q^{(1+\eps^2)h_\alpha c^l}(f_l-j_k, x+\eps^2 h_\alpha c^l,
y+\eps^2 h_\alpha c^l),
\]
and \eqref{PortStoneIneqU}, \eqref{PortStoneIneqL} give that the
above maximum is bounded above by a constant (that depends only on
$\eps$) as long as
\[
\frac{f_l-j_k}{(1+\eps^2)^2h_\alpha^2 c^{2l}}\ge t_0(\eps).
\]
This holds for $k, l$ satisfying \eqref{lkdistance} provided that
$\Delta$ is large enough.
\end{pf}

\subsection*{From geometric sequences to the full family}
We will now show that Propositions~\ref{geomSequences1},
\ref{geomSequences2} imply the result for the full family. As
noted in \citet{VE}, this can be done easily using the scaling
properties of the rate function $I$ and the regular variation of the
scaling factor $a^2\log\log a$ in \eqref{zPath}.

\begin{proposition} \label{fullnet} With probability one, every
sequence $(m(Z_{t_k}))_{k\ge1}$ with $t_k\to\infty$ has a convergent
subsequence, and the set of all possible limit points is exactly
\[
\C{K}:=\{\mu\in\C{M}\dvtx  I(\mu)\le1\}.
\]
\end{proposition}

\begin{pf}
For a measure $\mu\in\mathcal M$ and $a>0$, let $\mu_a\in\mathcal
M$ denote
the rescaled version of $\mu$ defined on every product of
measurable sets $H\times X\subset[0, \infty)\times\mathbb R$ as
\[
\mu_a(H\times X)=a   \mu(a^{-1}H \times a^{-2}X).
\]
Then for all continuous functions $\psi\dvtx [0, \infty)\times\mathbb R
\to\mathbb{R}$ with compact support,
%
\begin{equation} \label{scaledMeasure} \int\psi(h, x)\, d\mu_a(h,
x)=a\int\psi(ah, a^2 x) \, d\mu(h, x).
\end{equation}
Also, $f_{\mu_a}(t)=a^2f_\mu(t/a)$ for all $t\ge0$, and the
analogous statement holds for~$g_{\mu_a}$. Consequently, $I(\mu
_a)=I(\mu)$.

 For $t>1$, let $\mu^{(t)}=m(Z_t)$.
We will use the following claim.

\begin{claim*} $\!\!$If $\mu\!\in\!\mathcal M$, and for two sequences
$(t_k)_{k\ge1}, (p_k)_{k\ge1}$ with $\lim_{k\to\infty} t_k\!=\infty
, \lim_{k\to\infty} t_k/p_k=a\in(0, \infty)$ we have $\mu
^{(p_k)}\to\mu$, then $\mu^{(t_k)}\to\mu_{1/a}$.
\end{claim*}

\begin{pf} Take $\psi\dvtx [0, \infty)\times\mathbb
{R}\to\mathbb{R}$ continuous of compact support. For $k\ge1$, we may
assume that $p_k, t_k\ge1$, and let
\[
a_k:=\frac{t_k}{p_k}, \qquad  \beta_k:=\frac{\log\log p_k}{\log\log t_k}.
\]
Then $\int\psi(h, x)\, d \mu^{(t_k)}(h, x)$ equals
\begin{eqnarray*} \int_0^\infty\psi \biggl(t, \frac{x_B(a_k p_k
t)}{a_k^2 p_k^2 \log\log t_k} \biggr)\, dt&=&a_k^{-1} \int_0^\infty
\psi\biggl (a_k^{-1} s, a_k^{-2} \beta_k \frac{x_B(p_k s)}{p_k^2 \log
\log p_k} \biggr)\, ds\\  &=&a_k^{-1} \int\psi (a_k^{-1} h,
a_k^{-2} \beta_k x  )\, d \mu^{(p_k)}(h, x)
\\
  &=&aa_k^{-1}\int\psi (aa_k^{-1} h, a^2 a_k^{-2} \beta_k x
)\, d \mu^{(p_k)}_{1/a}(h, x).
\end{eqnarray*}
The last equality follows form \eqref{scaledMeasure}.
Since $a a_k^{-1}\to1$ and $\mu^{(p_k)}_{1/a}\to\mu_{1/a}$, it
suffices to prove that
\[
\int\psi (aa_k^{-1} h, a^2a_k^{-2} \beta_k x  )\, d \mu
^{(p_k)}_{1/a}(h, x)- \int\psi (h, x  )\, d \mu
^{(p_k)}_{1/a}(h, x)\to0
\]
as $k\to\infty$.

Let $C$ be the set of points in $[0, \infty)\times\mathbb{R}$ with
Euclidean distance at most 1 from the support of $\psi$, and $h^*$ the
maximum of the first projection of $C$.
For $\eps>0$, using the compactness of $C$, the uniform continuity of
$\psi$, and the fact that $\mu^{(p_k)}_{1/a}\in\mathcal{M}$, we
obtain that for large $k$, the absolute value of the last difference is
bounded from above by
\[
\int_C \varepsilon d \mu^{(p_k)}_{1/a}(h, x)\le\varepsilon h^*.
\]
This proves the claim.
\end{pf}

Now let $t_k\to\infty$ be a sequence. We fix a $c>1$, and write this
sequence uniquely as $t_k=a_kc^{i_k}$ with integers $i_k$ and real numbers
$a_k\in[1,c)$.

Regarding the first assertion of the proposition, note that by
Proposition~\ref{geomSequences1}, $m(Z_{c^{i_k}})$ has limit
points in $\C{K}$. Pick one, say $\mu$, and then passing to a
further subsequence along which $a_k$ converges to some limit
$a(c)\in[1, c]$, we see using the claim above that $\mu_{1/a(c)}$ is
a limit point along the
sequence~$t_k$.

For the second assertion, we have by Propositions
\ref{geomSequences1},~\ref{geomSequences2} that almost surely, all
limit points along $c^k$ are exactly the elements of the set
$\mathcal K$. It remains to show that along the above sequence
$(t_k)_{k\ge1}$, we do not get limit points outside $\C{K}$. If\vadjust{\goodbreak}
$\mu'$ is a limit point along $t_k$, then as above, we pass to a
further subsequence along which $a_k$ converges to some limit
$a(c)\in[1, c]$. It follows again from the claim above that along
$c^{i_k}$, $\mu'_{a(c)}$ is a
limit point. By Proposition~\ref{geomSequences1}, we have
$I(\mu'_{a(c)})\le1$. So that $I(\mu')=I(\mu'_{a(c)})\le1$, that is,
$\mu'\in\C{K}$.
\end{pf}

\section{The limit points of the motion}\label{sec7} \label{MotLIL}

The continuous time and space analogue of Sinai's walk is
diffusion in random environment, that is, the diffusion $X$ with
$X(0)=0$ that satisfies the formal differential equation
%
\begin{equation}\label{dre}d X(t)=d\beta(t)-\tfrac{1}{2}
V'(X(t)) \di t.
\end{equation}
Here, $\beta$ is standard Brownian motion, and $V$, the
environment, is a random function we pick before running the
diffusion. For the rigorous definition of this diffusion, as well
as its relation with Sinai's walk; see \citet{SHI}, \citet{SEI}.

In this work, we will consider diffusions run in a Brownian-like
environment. That is, we require from the measure governing $V$ to
be such that there is, on a possibly enlarged probability space, a
standard two-sided Brownian motion $B$ such
that for all $n\ge1$, we have
%
\begin{equation}\label{closetoBM}
\PP\Bigl(\sup_{|x|\le n}|V(x)-B(x)|\ge C_1 \log n\Bigr)\le\frac{1}{n^{C_2}}
\end{equation}
for some constants $C_1, C_2$. For these environments, the
diffusion does not explode in finite time. Moreover, its behavior
is dominated by the environment, and one aspect of this
phenomenon is captured by the following result. Recall the
definition of $x_B$ from Section~\ref{stepSection}.

\begin{theorem}[{[\citet{HU}, Theorem 1.1]}]
\label{Hu Theorem} Assume that $V$ satisfies
\eqref{closetoBM}. For every $\gd_1>0$, there exists $C,
t_0>0$ so that for $t\ge t_0$ and $\gl\ge1$, we have
%
\begin{equation} \label{Hu}
\pr\bigl(|X(t)-x_B(\log t)|>\gl\bigr)\le C\biggl (\frac{\log\log
t}{\sqrt{\gl}}+\frac{1}{(\log t)^{1-\gd_1}} \biggr).
\end{equation}
\end{theorem}

\subsection*{The limit points of the diffusion}
For the diffusion defined by \eqref{dre}, where $V$
satisfies \eqref{closetoBM}, we have the following analog
of Theorem~\ref{FunLIL}.

\begin{theorem}
\label{LILforDiffusion} With probability 1, the limit
points, as $a\to\infty$, in the topology of local weak
convergence of the graph occupation measures of the random
functions
\[
y_a:=\frac{X(e^{at})}{a^2\log\log a} , \qquad t\ge0,
\]
constitute the set
\[
\C{K}:=\{ \mu\in\C{M}\dvtx  I(\mu)\le1\}.
\]
Also, there is at least one limit point along every sequence $a_n\to
\infty$.\vadjust{\goodbreak}
\end{theorem}

\begin{pf}
Note that if $f_n,g_n$ is a sequence of functions on $[a, b]$ with
$m(f_n) \to\mu$ and $f_n-g_n \to0$ in (Lebesgue) measure, then
$m(g_n)\to\mu$. Indeed, assuming that for bounded, uniformly
continuous $\varphi$ we have
\[
\int_a^b \varphi(t,f_n(t)) \,dt \to\int_{[a, b]\times\R} \varphi
\, d\mu,
\]
we break down the integral to the set where $|f_n(s)-g_n(s)|<\eps$
and its complement. Since $\varphi$ is uniformly continuous, on
this set the integrand is close to $\varphi(t,g_n(t))$, while the
measure of the complement of this set is small. Thus $m(g_n)\to
\mu$.

Using this observation, Proposition~\ref{fullnet} and the
definition of the topology of $\C{M}$, it is clear that to prove
Theorem~\ref{LILforDiffusion}, it suffices to show that for every
$0<\eps<M<\infty$, as $a \to\infty$ we have
\[
\mathcal L\{s\in[\eps, M]\dvtx  |X(e^{as })-x_B(as)|>\eps
a^2\log\log a\} \to0.
\]
We prove this for $0<\eps<M=1$, as this is in no way different than the
general $M$ case. By changing variables $w=as$, the above quantity
will become
\begin{eqnarray*}
&&\int_{a \eps}^a \frac{\one(|X(e^{w})-x_{B}(w)|>\eps
a^2\log\log a )}{a} \,dw\\
&& \qquad  \le\int_{a\eps}^{\infty}\frac{
\one (|X(e^{w})-x_{B}(w)|>\eps  w^2\log\log w )
}{w}\,dw.
\end{eqnarray*}
If the last integral is finite for some $a>0$, then it
converges to $0$ as $a\to\infty$. Its expectation is
bounded using Theorem~\ref{Hu Theorem}, provided $a$
satisfies $\eps a>\log t_0$ and $(\eps a)^2\log\log(\eps
a)>1$, by
\begin{eqnarray*}
&&\int_{a \eps}^{\infty}\frac{1}{w}
\pr\bigl(|X(e^{w})-x_{B}(w)|>\eps  w^2\log\log w \bigr) \,dw \\
&& \qquad \le
\int_{a\eps}^{\infty} \frac{c}{w}
 \biggl(\frac{\log w}{\sqrt{w^2\log\log w}}+\frac{1}{w^{1-\gd
_1}} \biggr)\,dw <\infty.
\end{eqnarray*}
So that the integral is finite with probability 1.
\end{pf}

\subsection*{The limit points of the walk} \label{walkLIL}
To prove Theorem~\ref{FunLIL}, we will embed the walk it in a diffusion
generated by an appropriate random environment $V$.

Let $(S_n)_{n\ge1}$ be Sinai's walk with $\operatorname{Var}(\log((1-p_1)/p_1))=1$.
Define the step potential $V$ as follows: $V(0)=0$, and for every $n\in
\Z$, $V$ is constant
in $[n-1, n)$, and jumps at $n$ by $V(n)-V(n-)=\log((1-p_n)/p_n)$.
This potential
can be placed on a possibly enlarged probability space with a two-sided
Brownian motion $B$ so that \eqref{closetoBM} is satisfied.
This follows from the strong approximation theorem of
Koml\'os--Major--Tusn\'ady [Theorem 1 in \citet{KMT}]. The theorem
requires that $Y:=\log((1-p_1)/p_1)$ has $\ev(e^{\gl Y})<\infty$ for
$\gl$ in a neighborhood of 0, which is exactly the assumption we
made for the law of $p$ in the \hyperref[sec1]{Introduction}.

The walk can be embedded in the diffusion $X$ run in the
environment $V$ as follows. Let $t_0=0$ and
$t_n=\inf\{t>t_{n-1}\dvtx  |X(t)-X(t_{n-1})|=1\}$ for $n\ge1$.

\begin{theorem}[{[\citet{HS}, Proposition 9.1]}] \label{HS
Theorem}$(X(t_n))_{n\ge1}$ has the
same law as $(S_n)_{n\ge1}$. Moreover, $\{t_{n+1}-t_n\dvtx  n\ge1\}$
are i.i.d. with distribution that of the first hitting time $T$
of 1 for reflected standard Brownian motion.
\end{theorem}

We will need the fact that the law of $1/T$ has exponential tails.
This holds since if $1/T>x>0$, then the maximum or the negative of
the minimum of Brownian motion on $[0,1/x]$ is at least 1. Since
the maximum has the same distribution as $|B(x)|$, we have
%
\begin{equation}\label{1/T}
P(1/T>x) \le2P\bigl(|B(1/x)|>1\bigr) = 4P\bigl(B(1)>\sqrt{x}\bigr) \le c
e^{-x/2}.
\end{equation}

\begin{pf*}{Proof of Theorem~\ref{FunLIL}}
Let $t( \cdot )$ be the the piecewise linear continuous
extension of $t_n$ so that $t(n)=t_n$. To prove the theorem, it
suffices to show that for every $0<\eps<M<\infty$, as $a\to\infty$,
\[
\frac{S(e^{as})-x_{B}(as)}{a^2\log\log a} \to0\qquad  \mbox{in
measure on $[\eps,M]$, a.s.}
\]
Since $|S(s )-X(t(s))|\le1$, it suffices to show the
previous claim with $X(t(e^{as}))$ instead of $S(e^{as})$.
We break this down into two parts, namely
%
\begin{equation}\label{two conv in
measure} \frac{X( t(e^{as}))-x_B(\log t(e^{as}))}{a^2\log\log a} \to
0, \qquad\frac{x_B(as)-x_B(\log t(e^{as}))}{a^2\log\log a} \to0
\end{equation}
in measure on the interval $[\eps, M]$ as $a\to\infty$. We
assume for simplicity that $M=1$. Proceeding the same way
as for the diffusion, for the first claim it suffices to
prove that
\[
\int_{\eps a}^{\infty}\frac{1}{w} \one\bigl(|X(t(e^{w}))-x_B(\log
t(e^w))|>\eps  w^2\log\log w\bigr) \,dw<\infty
\]
for some $a>0$. The function $t$ has derivative equal to
$t_n-t_{n-1}$ in $(n-1, n)$, and undefined in $n$ for every
positive integer $n$. The integral over the $w$'s with
$1/t'(e^w)>\log w$ has expectation bounded above by
\[
\ev\int_{\eps a}^{\infty}\frac{1}{w} \one\bigl(1/t'(e^w)> \log w\bigr) \,
dw= \int_{\eps a}^{\infty}\frac{1}{w} \pr\bigl(1/t'(e^w)>\log w\bigr) \,
dw,
\]
which is finite because $1/t'$ has the same distribution as $1/T$ in
\eqref{1/T}. For the rest of the integral, we change variables $r=\log
t(e^w)$ and reduce the problem to the
finiteness of
\[
\int_{a\eps}^{\infty}\frac{1}{w} \frac{t(e^w)}{e^w}\frac{
\one(1/t'(e^w)\le\log
w)}{t'(e^w)}\one\bigl(|X(e^{r})-x_B(r)|>\eps  w^2\log\log w\bigr)
\,dr.
\]
By the law of large numbers and the fact that $\ev T=1$, we
have $t(s)/s \to1$. So
$t(e^w)/e^w\to1$ as $r\to\infty$, which
shows that $w/r\to1$ as well. Thus the above integral is finite
if
\[
\int_{a\eps}^{\infty}\frac{\log r}{r} \one\bigl(|X(e^{r})-x_B(r)|>(\eps/2)
r^2\log\log r\bigr) \,dr
\]
is finite. This follows by taking expectations and using Theorem
\ref{Hu Theorem}.

For the second convergence claim in \eqref{two conv in measure},
it suffices to show that
\[
\int_{2}^{\infty}\frac{1}{w} \one\bigl(x_B(w)\not=x_B(\log t(e^w))\bigr)
\,dw<\infty.
\]
Note that $x_B(w)\not=x_B(\log t(e^w))$ implies that $x_B$ has a
jump between $w$ and $\log t(e^w)$. By the law of large numbers,
for all $w$ large, this interval is contained in $(w-1, w+1)$. So
it suffices to show the finiteness of the integral
\[
\int_{2}^{\infty}\frac{1}{w} \one\bigl( x_B \mbox{ has a jump in }
(w-1, w+1)\bigr) \,dw.
\]
Applying Lemma~\ref{b jumps}, we bound its expectation from
above by
\[ c\int_{2}^{\infty}\frac{1}{w}
\log\frac{w+1}{w-1} \,dw<\infty.
\]
\upqed\end{pf*}

In the proof of Theorem~\ref{FunLIL}, we use the next lemma,
which gives a bound on the probability that $x_B$ jumps on an
interval.

\begin{lemma}\label{b jumps}
The process $x_B$ satisfies $ \pr(x_{B}(s)\not=x_B(t))\le
c |\log(t/s)|$ for some finite constant $c$ and all $t,s>0$.
\end{lemma}

\begin{pf}
This holds because the jumps of $x_B(e^t)$ form a
translation-invariant point process on $\mathbb R$ with finite
mean density $c$. Rather than proving this, we will invoke the
exact formula for the above probability. Assuming that $s<t$, and
using the scaling property of $x_B$ [see \eqref{xBScaling}], the
probability in question equals $\pr(x_{B}(1)\not=x_B(t/s))$.
However,
\[
\pr\bigl(x_{B}(1)=x_B(t/s)\bigr)= \biggl(\frac{t}{s} \biggr)^{-2}\frac
{5-2e^{-(t/s)+1}}{3}
\]
as is shown in the proof of Theorem 2.5.13 in \citet{ZE}. And this
gives easily the required bound.
\end{pf}

\begin{pf*}{Proof of Corollary~\ref{c.application}} In fact, we will
prove that if $\gamma\dvtx [0,1]\to[0, \infty)$ is differentiable with
$(t\mapsto t^3 \gamma(t))$ nondecreasing, then
%
\begin{equation}\label{corollaryGeneralization}\limsup_{a\to\infty}
\frac{1}{a^2\log\log a}\int_0^1 \gamma(t) S(e^{at})\, dt =\frac
{4}{\pi^2} s_0^3 a(s_0),
\end{equation}
where $s_0$ is any root of $2\int_s^1 \gamma=s\gamma(s)$ in $(0, 1)$.\vadjust{\goodbreak}

For $H\dvtx [0,\infty) \times\D{R}\to\D{R}$ with compact support and
whose projection of the set of the discontinuity points in the
$x$-axis has Lebesgue measure zero, the map $(\C{M}\ni\mu\mapsto
\int H\di\mu)$ is continuous in the weak topology, because any
$\mu\in\C{M}$ has first projection Lebesgue measure. Combining
this with the definition of the graph occupation measure, we get
%
\begin{eqnarray}\label{weakConvergence}&&\limsup_{a\to\infty} \int
_0^\infty H \biggl(t, \frac{S(e^{at})}{a^2\log\log a} \biggr)\, dt
\nonumber
\\[-8pt]
\\[-8pt]&& \qquad =\sup\biggl \{\int H(x, y)\di\mu(x, y)\dvtx  \mu\in\C{M}, I(\mu)\le
1 \biggr\}.
\nonumber
\end{eqnarray}
Everywhere below, we use the abbreviation $A:=8/\pi^2$.

For the choice $H(x, y):=\gamma(x)  y  \mathbf{1}_{x\in[0, 1],
|y|\le A+1}$ the limits in equations \eqref{corollaryGeneralization},
\eqref{weakConvergence} agree because by Theorem 1.3 in \citet
{HS}, it holds
\[
\operatorname{\varlimsup}\limits_{n\to\infty} \frac{\max_{1\le k\le n} |S_k|}{(\log
n)^2\log\log\log n}=A.
\]
It remains to evaluate the supremum in \eqref{weakConvergence} for
this choice of $H$. If $\gamma$
is identically zero, the corollary holds trivially. So we assume
that $\gamma$ is positive somewhere in $[0,1]$. Since $\gamma$ is
nonnegative, it follows from the form of the rate function $I$
that the above supremum equals
\[
\sup \biggl\{\int_0^1 \gamma(t) f(t)\di t\dvtx  f(0)=0, f \mbox{ nondecreasing}, \int_0^1 t^{-2} \di f(t)\le A \biggr\}.
\]
We did not include the factor $\mathbf{1}_{|f(t)|\le A+1}$ inside the
integral because the conditions on $f$ imply that $0\le f(t)=\int_0^t
d f(s)\le\int_0^1 s^{-2} \di f(s)\le A$.

For a given nondecreasing $f\ge0$, define $F(t):=\int_0^t s^{-2}
\di f(s)$, so that $f(t)=\int_0^t s^2 \di F(s)$. We use this
representation of $f$ and apply first Fubini's theorem and then
integration by parts in $\int_0^1 \gamma(t) f(t)\di t$ to write it as
\[
\int_0^1 F(s)r'(s) \di s,
\]
where $r(s):=-s^2\int_s^1 \gamma(t) \di t$. Using the fact that
$(t\mapsto t^3 \gamma(t))$ is nondecreasing in $[0, 1]$, we find that
$r'$ is nonpositive before $s_0$ and nonnegative after $s_0$. And since
$F\le A$, the above integral is bounded above by
\[
-A r(s_0)=\frac{A}{2}s_0^3 \gamma(s_0).
\]
The last equality follows from $r'(s_0)=0$.

For the choice $f^\gamma(t)=s_0^2 A
\mathbf{1}_{(s_0, \infty)}$, we get $\int_0^1 \gamma(t) f^\gamma
(t) \di t=A
s_0^2\int_{s_0}^1 \gamma(t) \di t=A s_0^3\gamma(s_0)/2$, so that the
supremum is
achieved. Clearly, the only measure that achieves the supremum is
the element of $\C{M}$ that puts all its mass on the graph of
$f^\gamma$.\vadjust{\goodbreak}

Finally, for the case of the corollary, $\gamma(t)=t^r$, we compute
$s_0=\eta_r:=(\frac{2}{r+3})^{1/(r+1)}$, while the value of the
supremum is $(A/2)\eta_r^{r+3}$.
\end{pf*}

\section{The probability of confinement}\label{sec8} \label{BrownianComputations}
In this section, we compute the asymptotic decay of the
probabilities that Brownian motion or Brownian motion reflected
from its running minimum stay on certain bounded sets for large
intervals of time.

Fix $h>0, x\in(0, h), t>0$, and let $Q^h(t, x,  \cdot  )$ be
the density of the measure
\[
S\mapsto\PP_x\bigl(B_t\in S, B[0, t]\subset(0, h)
\bigr).
\]
Proposition 8.2 in \citet{PS} gives
%
\begin{equation}\label{PortStone}Q^1(t, x, y)=2\sum_{n=1}^\infty
e^{-n^2\pi^2 t/2}\sin(n\pi x)\sin(n\pi
y).
\end{equation}
Using this and Brownian scaling, we get that there exists a
universal constant $c_2$ so that for all $t\ge1$,
$x,y\in[0,h]$, we have
%
\begin{equation}\label{PortStoneIneqU}
Q^h(t, x, y)\le c_2 h^{-1}
\exp \biggl(-\frac{\pi^2}{2}\frac{t}{h^2} \biggr).
\end{equation}
Moreover, for every $\eps>0$ there exists a constant
$c_1=c_1(\eps)$ so that for all $t\ge1$, and $x,y\in[\eps
h,(1-\eps)h]$, we have
%
\begin{equation}\label{PortStoneIneqL}
Q^h(t, x, y) \ge
c_1(\eps)h^{-1}\exp \biggl(-\frac{\pi^2}{2}\frac{t}{h^2} \biggr).
\end{equation}

Recall from \eqref{runningMin} the notation for the past
minimum of a given process, and from \eqref{RBM} the
process $R=B-\underline B$. The probability that $R$ stays
confined in an interval for a large time interval $[0, t]$
decays exponentially in $t$. In the next lemma, we compute
the exact rate of decay.

\begin{lemma}\label{confinementCost}
For $K>0$, $\eps\in[0, 1/2)$, $w\in[0, 1)$ and $z\in(0,1)$,
\begin{eqnarray*}
(\mathrm{a})&&
\lim_{t\to\infty} \frac{1}{t}\log\PP_z\bigl(B([0, t])\subset[0,1],
B(t)\in[\eps, 1-\eps]\bigr)=-\frac{\pi^2}{2},
\\
(\mathrm{b})&&   \lim_{t\to\infty} \frac{1}{t}\log
\PP\bigl(R([0,t])\subset[0,1], R(t)\in[0, 1-\eps]   |  R(0)=w\bigr)=-\frac
{\pi^2}{8},
\\
(\mathrm{c}) && \lim_{t\to\infty} \frac{1}{t}\log\PP_0 \bigl(R([0,
t])\subset[0,1], \ul{B}(t)\ge-K  \bigr)=-\frac{\pi^2}{2}.
\end{eqnarray*}
\end{lemma}

For $\eps_1\in(0,1/2)$ fixed, the convergence in
(\BMrestricted) is uniform over \mbox{$z\in[\eps_1,
1-\eps_1]$}, and the convergence in (\restrictionNoFloor) is
uniform over $w\in[0,1-\eps_1]$.

Comparing (\restrictionNoFloor) and (\restrictionFloor),
note the drastic effect of the restriction $\ul{B}(t)\ge
-K$. The process $(R, -\ul{B})$ has the same law as $(|B|,
L)$ where $L$ is the process of local time at zero\vadjust{\goodbreak} for the
Brownian motion $B$. Phrased in terms of $(|B|, L)$, the
first event requires $B([0, t])\subset[-1, 1]$, the other
requires additionally that $L_t^0\le K$, that is, $B$ does not
hit zero many times. This restriction makes the second
event more like $B([0, t])\subset(0, 1]$, that is, $B$ is
essentially restricted to an interval of half size than
before.

\begin{pf}
(a) Follows by integrating \eqref{PortStoneIneqU} and
\eqref{PortStoneIneqL} over $y$.

(b) Since $R$ has the same law as the absolute value
of Brownian motion, the claim follows by the scaling
property of Brownian motion and (\BMrestricted).

(c) \textit{Lower bound:} Pick an open interval $V$ of
length 1 around 0 that does not contain $-K$. Then
\[
B([0, t])\subset V   \quad  \mbox{implies} \quad   R([0, t])\subset[0,1], \qquad  \ul
{B}_t \ge-K.
\]
By (a) applied with $\eps=0$, the probability of
the first event decays like $\exp (-t[\pi^2/\allowbreak 2 +o(1)] )$
as $t\to\infty$.

(c) \textit{Upper bound:} Let $A_t$ denote the event in
question. Subdivide the rectangle $[0,t]\times[-K,0]$ into
$n\times n$ small isomorphic rectangles. Each rectangle is
a product of a time interval ${\mathcal T_i}:=[(i-1)t/n, it/n]$ with
$i=1, 2, \ldots, n$ and a space
interval. Consider the graph of the process $\underline
B$, and let $J$ be the union of the subdivision rectangles
it intersects. Fix $m\ge1$. When $\underline B(t)\ge-K$,
we have:
\begin{itemize}
\item$J=\bigcup_{i=1}^n {\mathcal T_i}\times
{\mathcal B_i}$ for some space intervals $\mathcal B_i$.
\item Let $\mathcal N=\{i\dvtx \mbox{length } \mathcal B_i\le(m+2)K/n\}$.
Then $|\mathcal N|\ge(1-1/m)n$.
\end{itemize}
The first claim is clear. For the second, note that on the
time intervals $\mathcal T_i$ for $i\notin\mathcal N$ the
process $\underline B$ decreases by at least $mK/n$. But
the total decrease is at most $K$, so there are at most
$n/m$ such indices $i$. We have
\begin{eqnarray*}
P(A_t) &=& P \bigl(\underline B(t)\ge-K, \mbox{ graph }
B[0,t]\subset\mbox{graph }\underline B[0,t]+\{0\}\times[0,1] \bigr)
\\ &\le&\sum_J \PP(\mbox{graph }
B[0,t] \subset J+\{0\}\times[0,1]).
\end{eqnarray*}
Here the sum is over all unions $J$ of rectangles
satisfying the conditions above. By the Markov property, we
get the upper bound
\begin{eqnarray*}
&&\sum_J
\prod_{i=1}^n \max_x \PP_x\bigl(B({\mathcal T_i}) \subset
{\mathcal B_i}+ [0,1]\bigr)\\
&& \qquad \le2^{(n^2)} \max_x
 \bigl(\PP_x\bigl(B[0,t/n]\subset[0,1+(m+2)K/n]\bigr) \bigr)^{|\mathcal
N|}.
\end{eqnarray*}
The inequality follows by considering only the
indices $i\in\mathcal N$. Brownian scaling, part (a) with $\eps=0$
and the fact $|\mathcal N|\ge
n(1-1/m)$ gives
\[
\limsup_{t\to\infty} \frac{1}{t}\log\PP(A_t) \le
-\frac{1}{n} \biggl(\frac{\pi^2}{2}\frac{1}{(1+(m+2)K/n)^2} \biggr)n(1-1/m).
\]
Since this holds for $m,n$ arbitrary, we let
$n=m^2\to\infty$ to get the desired upper bound.\vadjust{\goodbreak}
\end{pf}

Below, we will use the operator $R_x$ of reflection from
the past infimum, defined in \eqref{reflectionFromInf}, and
the notation $R=R_0 B$ from \eqref{RBM}.

For $0\le w\le x<y$ and $0<h_1<h_2$, call $\Gamma(w, x, y, h_1, h_2)$
the set
\begin{eqnarray*}
 &&\{R_w f\in C(x, y, h_1)\cap H^R(y, h_1)\cap B(y, h_2) \}\\
 && \qquad {}\cap
 \bigl\{f-f\bigl(x(1+\gd)\bigr) \le\eps^2 \mbox{ on } [x(1+\gd), y(1+\gd
)] \bigr\},
\end{eqnarray*}
which is involved in the definition of $\mathcal R(\mathbf{h}, \mathbf
{x}, \gd,
\eps)$. More precisely, the set $E_i$ corresponding to an index
$i\in\C{I}_\infty$, defined in Section~\ref{specialSets}, is
exactly the set
\[
\Gamma\bigl(w_q(1+\delta), w_i, x_i, h_i, h_{i^+}\bigr).
\]

The following lemma computes for large $M$
the probability that the scaled Brownian
motion $B(M   \cdot  )$ is in the special sets $C, H,
B, \Gamma$, defined in Section~\ref{specialSets} and above.

\begin{lemma} \label{blocksCosts}
Let $0\le x<y, 0<h_1<h_2, h>0$, and small enough $\eps, \delta>0$.

Uniformly on $z \in[0, h-\eps h], z_1 \in[0,
h_1-\eps h_1], z_2\in\D{R}$, as $M\to\infty$ we have:
\begin{eqnarray*}
(\mathrm{a})  &&\frac{1}{M}\log\PP\bigl(B(M   \cdot )\in
C(x,y,h)  |  B\bigl(M x(1+\gd)\bigr)=z \bigr)  \to
-\frac{\pi^2}{2}\frac{y-x-\gd(x+y)}{h^2(1+\eps^2)^2};
\\
(\mathrm{b})&& \frac{1}{M}\log\PP\bigl(B(M  \cdot )\in H(y,h)  |
B\bigl(M y(1-\gd)\bigr)=z\bigr)  \to-\frac{\pi^2}{2}\frac{\gd y}{h^2(1+\eps
)^2};
\\
(\mathrm{c}) && \frac{1}{M}\log\PP\bigl(B(M  \cdot )\in
B(y,h_2)  |  B(M y)=z_1\bigr)
 \to-\frac{\pi^2}{2}\frac{\gd y}{h_2^2(1+\eps+\eps^2)^2}; \\
(\mathrm{d}) &&   \frac{1}{M}\log
\PP \biggl(B(M  \cdot )\in\Gamma   \bigg|
\begin{array}{l}
R\bigl(Mx(1+\gd)\bigr)=z_1\\
B\bigl(Mx(1+\gd)\bigr)=z_2
\end{array}
 \biggr)\\
 && \qquad  \to
-\frac{\pi^2}{8} \biggl( \frac{y-x-\gd x}{h_1^2}+ \frac{\gd
y}{h_2^2(1+\eps)^2} \biggr),
\end{eqnarray*}
where $\Gamma:=\Gamma(w, x, y, h_1, h_2)$ and the events
$C,H,B,\Gamma$ depend on $\eps, \delta$.
\end{lemma}

\begin{pf}
(a) It follows from Lemma~\ref{confinementCost}(\BMrestricted) and
the scaling property
of Brownian motion.

(b) The exponential rate of decay of the event in question is the
same as the one of
confinement on $[-\eps h, h]$ between times $y(1-\gd), y$ and
ending in $[0, (1-\eps)h]$. Because the difference of the two events
is contained on the event of confinement on the smaller
interval $[-\eps h+\eps^2 h, h]$, which decreases exponentially
faster. Thus, the result follows again from Lemma \ref
{confinementCost}(\BMrestricted).

(c) The same reasoning as in part (b) proves this claim too.

(d) We can assume that $w=0$. Then we let $\Gamma(x, y, h_1,
h_2)=\Gamma(0, x, y, h_1,
h_2)$, and $\Gamma'(x, y, h_1, h_2)$ to be only the first set in the
intersection defining $\Gamma(x, y,\allowbreak h_1,
h_2)$; that is, we remove the restriction
$f-f(x(1+\gd)) \le\eps^2$ on $[x(1+\gd),\break y(1+\gd)]$. We first
prove the claim with $\Gamma'$ in
place of $\Gamma$. To this aim, we observe that
%
\begin{eqnarray}\label{RBC}
&&\lim_{M\to\infty}\frac{1}{M}\log\PP\bigl(R(M  \cdot  )\in C(x, y,
h_1)    | R\bigl(M x(1+\gd)\bigr)=z_1\bigr)\nonumber
\\[-8pt]
\\[-8pt]
&& \qquad =-\frac{\pi^2}{8}\frac{y-x-\gd(x+y)}{h_1^2},\nonumber\\ \label{RBH}
&&\lim_{M\to\infty} \frac{1}{M}\log\PP\bigl(R(M  \cdot )\in H^R(y,
h_1)   |  R\bigl(My(1-\gd)\bigr)=z_1\bigr)= -\frac{\pi^2}{8}\frac{\gd
y}{h_1^2},
\\ \label{RBB}
&&\lim_{M\to\infty} \frac{1}{M}\log\PP\bigl(R(M  \cdot )\in B(y,
h_2)  |  R(My)=z_1\bigr)=
-\frac{\pi^2}{8}\frac{\gd y}{h_2^2(1+\eps)^2},
\end{eqnarray}
where the convergence is uniform over $z_1\in[0, h_1-\eps h_1]$.

The first expression follows from Lemma \ref
{confinementCost}(\restrictionNoFloor).
For the second, an upper bound is given by the same relation because
the event requires confinement
on $[0, h_1]$ for the time interval $[M y(1-\gd), M y\gd]$. For a
lower bound, we will consider two events whose intersection is
inside the event of interest and whose probability we will
estimate. The first event
\[
 \left\{
\begin{array}{c}
\mbox{In the time interval } [My(1-\gd), My(1-\gd)+1], \\
\mbox{$R$ \mbox{ visits } 0, stays in $[0, h_1]$, ends in } [0, h_1-\eps h_1]
\end{array}
 \right\}
\]
realizes the requirement of the visit to zero. Given that
$R(My(1-\gd))=z_1\in[0, h_1-\eps h_1]$, this event has a positive
probability independent of $M$. The second event is
\[
\left \{
\begin{array}{c}
\mbox{In the time interval } [My(1-\gd)+1, My] ,\\
R \mbox{ stays in $[0, h_1]$, ends in } [0, h_1-\eps h_1]
\end{array}
 \right\}.
\]
To compute the probability of the intersection, we apply the
Markov property at time $My(1-\gd)+1$. Then the probability of the
second event, conditioned on the value of $R$ at $My(1-\gd)+1$,
will decay exponentially as $M\to\infty$ with the same rate as if
$R$ was staying in $[0,h_1]$ in the slightly larger time interval
[$My(1-\gd), My]$ and was ending in $[0, h_1-\eps h_1]$. So that the
lower bound obtained for the left-hand side of \eqref{RBH}
coincides with the upper bound. Equation \eqref{RBB} is proved in
the same way.

Relation (\lastCost) with $\Gamma'$ is place of $\Gamma$ now follows
by applying the Markov property and using \eqref{RBC}, \eqref{RBH},
\eqref{RBB}.

To prove (\lastCost) itself, we note that the left-hand side increases
if we put
$\Gamma'$ in place of $\Gamma$. This observation, together with the
above, gives an upper bound, but we
can show a lower bound, too.

Let $\Delta_M$ be
the set of continuous functions $f$ on $[0, \infty)$ with
\begin{eqnarray*}
f\bigl(x(1+\gd)+M^{-1}\bigr)-f\bigl(x(1+\gd)\bigr)&\le&-h_2(1+\eps),\\
f(s)-f\bigl(x(1+\gd)\bigr)&<&\eps^2, \qquad  R f(s) \le  h_1-\eps h_1
\end{eqnarray*}
for $s\in[x(1+\gd), x(1+\gd)+M^{-1}]$, and $E_M$ the set
\[
C\bigl(x+\bigl(M(1+\gd)\bigr)^{-1}, y, h_1\bigr)\cap H^R(y, h_1)\cap B(y, h_2).
\]
Then
\[
 \{B(M   \cdot )\in\Delta_M \}\cap \{R(M   \cdot
)\in E_M \}\subset \{B(M  \cdot  )\in
\Gamma(x, y, h_1, h_2) \}.
\]
To see this, note that the inclusion holds with $\Gamma'$
in place of $\Gamma$. But because at $x(1+\gd)+M^{-1}$ the
process $B(M   \cdot )-B(M x(1+\gd))$ takes a value less
than $-h_2(1+\eps)$, and after that $R$ stays below
$h_2(1+\eps)$, it follows that $B(M   \cdot )-B(M
x(1+\gd))$ stays negative in the interval
$[x(1+\gd)+M^{-1}, y(1+\gd)]$. And of course it stays below
$\eps^2$ in $[x(1+\gd), x(1+\gd)+M^{-1}]$ because of
$\Delta_M$.

By applying the Markov property at time $Mx(1+\gd)+1$, we get that
the probability of the above intersection, conditioned on the
values of $R(Mx(1+\gd)), B(Mx(1+\gd))$ as in (\lastCost),
is at least the product of
%
\begin{equation}\label{yoda1}
\PP \biggl(B(M  \cdot )\in\Delta_M   \bigg |
\begin{array}{c}
R\bigl(Mx(1+\gd)\bigr)=z_1\\
B\bigl(Mx(1+\gd)\bigr)=z_2
\end{array}
 \biggr)
\end{equation}
and
%
\begin{equation}\label{yoda2}
\inf_{z_3\in[0, h_1-\eps h_1]} \PP\bigl(R(M   \cdot  )\in E_M|
R\bigl(Mx(1+\gd)+1\bigr)=z_3\bigr).
\end{equation}
The probability in \eqref{yoda1} is positive and does not depend
on $M$. The asymptotic decay as $M\to\infty$ for the
probability in \eqref{yoda2} is computed as in the case of $\Gamma'$. The
change in the restriction interval from $[Mx(1+\gd),
My(1-\gd)]$ to $[Mx(1+\gd)+1, My(1-\gd)]$ does not change
the result.
\end{pf}

\section{The topology of $\mathcal M$ and step functions}\label{sec9}
\label{s.topology}

This section contains the proofs of the topological lemmas
used in Sections~\ref{UpperBoundSection},
\ref{specialSets} and~\ref{EnvLIL} for the large
deviation principle and the functional law of the iterated logarithm
for the environment.

Lemma~\ref{densityBelow} is a consequence of Lemmas
\ref{NbhdAssumption} and~\ref{densityBelow3} below.

\begin{lemma} \label{NbhdAssumption}
Let $\mu\in\C{M}$ and $(\mathbf{h}, \mathbf{x})\in\C{S}$. Assume that
%
\begin{eqnarray} \label{nbhdAssumption}\qquad
\begin{tabular}{l}
$\mbox{there is }\BS{x}'\in\R^N \mbox{ so that }(h_i,
x_i')\in\supp(\mu) \mbox{ for all }i\in\C{I}   \mbox{ and }
x_i'>x_i $\\ $\mbox{if } x_i>0 \mbox{ and }x_i'<x_i \mbox{
if } x_i<0.$
\end{tabular}
\end{eqnarray}
Then the set ${\mathcal U}(\mathbf{h}, \mathbf{x}, \eps)$ defined in
\eqref{neighborhood} is a neighborhood of $\mu$ for every
$\eps>0.$
\end{lemma}

The proof is straightforward using the definition of weak
convergence, so we omit it.

\begin{lemma}\label{densityBelow3} For each $\mu\in\C{M}$ and
$A<I(\mu)$, there is $(\mathbf{h}, \mathbf{x})\in\C{S}$ satisfying
\eqref{nbhdAssumption} so that $I( \mathbf{h}, \mathbf{x})>A.$
\end{lemma}

\begin{pf}
We will abbreviate $f_\mu, g_\mu$ to $f,g$. We also remind
the reader that for a bounded function $F$, a nondecreasing
function $\alpha$, both defined on a finite closed interval\vadjust{\goodbreak}
$[a, b]$, and a partition
$\C{P}=\{a=:t_0<t_1<\cdots<t_n:=b\}$ of $[a, b]$, the lower
Stieltjes sum $L(\C{P}, F, \alpha)$ is defined as
\[
\sum_{i=1}^n \inf\{F(t) \dvtx  t\in[t_{i-1}, t_i]\}\bigl(\alpha(t_i)-\alpha
(t_{i-1})\bigr).
\]
We consider three cases for $\mu$.

\textsc{Case 1}. $0<s_{\mu-}<s_{\mu+}=\infty$.

We can write $A=(A_1+A_2)\pi^2/2+A_3\pi^2/8$ for some
$A_i$'s with
\[
\int_0^{s_{\mu-}}t^{-2} \di f(t)>A_1, \qquad
\int_0^{s_{\mu-}}t^{-2} \di g(t)>A_2, \qquad
\int_{s_{\mu-}}^H t^{-2} \di f(t)>A_3,
\]
where $H\in(s_{\mu-}, \infty)$ is large enough. Since
$f, g$ are left continuous at $s_{\mu-}$, we can find two
finite subsets $\C{P}_1, \C{P}_2$ of $[0,
s_{\mu-})$ so that $\C{P}_1 \cap\C{P}_2=\{0\}$, and when
considered as partitions of the intervals $[0,
\max\C{P}_1], [0, \max\C{P}_2 ]$, the corresponding lower
Stieltjes sums satisfy
%
\begin{equation}\label{firstInt}
L(\C{P}_1, t^{-2}, f)> A_1, \qquad
L(\C{P}_2, t^{-2}, g)> A_2.
\end{equation}
We can also find a finite subset
$\C{P}_3$ of $[s_{\mu-},H]$ containing $s_{\mu-}$, with
%
\begin{equation}\label{thirdInt}
L(\C{P}_3, t^{-2}, f)> A_3.
\end{equation}
We can assume that $f|\C{P}_1\cup\C{P}_3, g|\C{P}_2$ are
strictly increasing. In particular, $f(\zeta_1),
g(\zeta_2)>0$, with $\zeta_i:=\min(\C{P}_i\setminus\{0\})$
for $i=1,2$. We can also assume that
%
\begin{eqnarray} \label{supptAssumption}
 (h, f(h))&\in&\supp(\mu)  \qquad \mbox{for } h\in\C
{P}_1\cup\C{P}_3, \nonumber
\\[-8pt]
\\[-8pt]
(h, -g(h))&\in&\supp(\mu)   \qquad \mbox{for } h\in\C{P}_2.
\nonumber
\end{eqnarray}
If, for example, this is not the case for an $h\in
\C{P}_1$, we go as follows. The point
\[
h':=\sup\{\eta\le h\dvtx  (\eta, f(\eta))\in\supp(\mu)\}
\]
satisfies $h'<h$ by the assumption and the left continuity
of $f$, $f(h')=f(h)$ by the minimality property in the
definition of $f$, and $(h', f(h'))\in\supp(\mu)$. We can
find an $h''<h'$ near $h'$ such that $(h'', f(h''))\in
\supp(\mu)$, $h''\notin\C{P}_2$, and $f(h'')$ is as close
to $f(h)$ as we want, because $f$ is left continuous. Finally, we
replace $h$ with $h''$ in
$\C{P}_1$. The lower Stieltjes sum over the new partition is
larger than before because $t^{-2}$ is decreasing.

Also, for the $\eta_i:=\max\C{P}_i$ for $i=1,2$, we
can arrange that $\eta_1<\eta_2$ because $(s_{\mu-},
-g(s_{\mu-}))\in\supp(\mu)$.

Let $\BS{h}\in\mathbb{R}^N$ be the vector having coordinates the
elements of the set
$(\C{P}_1\cup\C{P}_2\cup\C{P}_3)\setminus\{0\}$ ordered as
$h_1<\cdots<h_N$, and define the vector
$\BS{x}(\eps)\in\R^N$ as
\[
x_i(\eps):=
\cases{\displaystyle
f(h_i)-\eps ,&\quad  if   $h_i\in\C{P}_1\cup\C{P}_3 $,\cr\displaystyle
-g(h_i)+\eps ,&\quad  if   $h_i\in\C{P}_2$,
}
\]
for all $i\in\{1,\ldots, N\}$ and all $\eps\in[0,
f(\zeta_1)\wedge g(\zeta_2))$.\vadjust{\goodbreak}

Using the notation of Section~\ref{stepSection}, we note that
for $\eps$ small enough, as above, all the pairs $(\BS{h},
\BS{x}(\eps))$ give rise to the same index set $\C{I}_\infty$, and
we have $\C{I}_\infty=\{i\dvtx  h_i \in\C{P}_3\}$ because $h_i\le
\eta_1<\eta_2$ for all $i$ with $h_i\in\C{P}_1$, and
$\eta_2=h_{i_0}$ with $x_{i_0}(\eps)<0$. Also
\[
I(\BS{h}, \BS{x}(\eps))=\frac{\pi^2}{2}\sum_{i\in\C{I}\setminus
\C{I}_\infty}\frac{|x_i(\eps)-x_{i^-}(\eps)|}{h_i^2}+\frac{\pi
^2}{8}\sum_{i\in\C{I}_\infty} \frac{|x_i(\eps)-x_{i^-}(\eps)|}{h_i^2}
\]
and
\[
\lim_{\eps\to0}I(\BS{h}, \BS{x}(\eps))=I(\BS{h}, \BS{x}(0))>A.
\]
The last inequality follows from \eqref{firstInt},
\eqref{thirdInt}, the equalities
%
\begin{eqnarray}
\sum_{i\in\C{I}\setminus\C{I}_\infty\dvtx x_i>0}
\frac{|x_i(0)-x_{i^-}(0)|}{h_i^2}&=& L(\C{P}_1, t^{-2},f), \\
\sum_{i\in\C{I}\dvtx  x_i<0}
\frac{|x_i(0)-x_{i^-}(0)|}{h_i^2}&=& L(\C{P}_2, t^{-2}, g),
\end{eqnarray}
in which we use that $f(0)=g(0)=0$, and the inequality
\[
\sum_{i\in\C{I}_\infty}
\frac{|x_i(0)-x_{i^-}(0)|}{h_i^2}\ge L(\C{P}_3, t^{-2}, f).
\]
The last inequality holds because the left-hand side equals exactly the
Stieltjes sum in the right-hand side plus the term corresponding to
$i:=\min\C{I}_\infty$.

Thus, for small $\eps$, the pair $(\BS{h},
\BS{x}(\eps))$ is in $\C{S}$, satisfies assumption
\eqref{nbhdAssumption} because of \eqref{supptAssumption}, and it has
$I(\BS{h},
\BS{x}(\eps))>A$.

\textsc{Case 2}. $s_{\mu-}=s_{\mu+}=\infty$.

There is $H>0$ finite with
\[
(\pi^2/2)\int_0^H t^{-2}
\di(f+g)(t)>A.
\]
Let $A_1, A_2$ be such that
$A=(\pi^2/2)(A_1+A_2)$ and
\[
\int_0^H t^{-2} \di f(t)>A_1,
\qquad\int_0^H t^{-2} \di g(t)>A_2.
\]

We find two finite subsets $\C{P}_1, \C{P}_2$ of $[0,
H]$,
so that $\C{P}_1 \cap\C{P}_2=\{0\}$, and when considered as
partitions of the intervals $[0, \max\C{P}_1], [0, \max\C{P}_2]$,
the corresponding lower Stieltjes sums satisfy %
%
\begin{eqnarray}\label{lowerSum21} L(\C{P}_1, t^{-2}, f)&>&A_1,
\\ \label
{lowerSum22}
L(\C{P}_2, t^{-2}, g)&>& A_2.
\end{eqnarray}
Again, we can assume that $f|\C{P}_1, g|\C{P}_2$ are
strictly increasing, $(h, f(h))\in\supp(\mu)$ for $h\in
\C{P}_1$, $(h, -g(h))\in\supp(\mu)$ for $h\in\C{P}_2$,
and $\eta_1<\eta_2$, where $\eta_i:=\max\C{P}_i$ for
$i=1,2$, as before.\vadjust{\goodbreak}

Pick a number $\eta_1'>\eta_2$ with $(\eta_1',
f(\eta_1'))\in\supp(\mu)$ (recall that $s_{\mu+}=\infty$),
and let $\C{P}_3:=\{\eta_2, \eta_1'\}$.

Let $\BS{h}\in\mathbb{R}^N$ be the vector having coordinates the
elements of the set
$(\C{P}_1\cup\C{P}_2\cup\C{P}_3)\setminus\{0\}$ ordered as
$h_1<\cdots<h_N=\eta_1'$, and the proof continues as in the
first case. Here we just note that in the resulting pairs
$(\BS{h}, \BS{x}(\eps))$, only one element belongs to the
final index set $\C{I}_\infty$, which is due to $\eta_1'$.
The presence of $\eta'_1$ is needed so that in the formula
for $I(\BS{h},\BS{x})$, all increments
$(x_i-x_{i^-})/h_i^2$ with $i\le{N-1}$ get coefficient
$\pi^2/2$, and this is enough to make $I(\BS{h}, \BS{x})$
larger than $A$ because of \eqref{lowerSum21},
\eqref{lowerSum22}.

\textsc{Case 3}. $0=s_{\mu-}<s_{\mu+}=\infty$.

In this case, we work only with the function $f$ and
one partition. The proof is similar to the previous case
and easier.

Since the roles of $f,g$ are symmetric, these are the
only truly different cases.
\end{pf}

Lemma~\ref{rateFunctionApprox} is an immediate consequence
of Lemmas~\ref{vidor} and~\ref{darthvader} below. The next
lemma essentially shows that the pairs $(\mu_{\mathbf{h}, \mathbf
{x}},I(\mathbf{h}, \mathbf{x}))$ are relatively dense in
$\{(\mu,I(\mu))\dvtx \mu\in\mathcal M\}$.

For $L>0$, let $\mathcal{T}_L$ be the topology of weak convergence on
compact subsets of $[0, L]\times\mathbb{R}$ for elements of $\mathcal
{M}$. Note that the topology of $
\mathcal{M}$ is the union of the family
$\{\mathcal{T}_L\dvtx L>0\}
$, which is increasing.

\begin{lemma}\label{vidor}   For every open $G\subset\mathcal M$,
$\mu\in
G$ with $I(\mu)<\infty$, and \mbox{$\delta>0$}, there exists $(\mathbf{h},
\mathbf{x})\in\C{S}$ and $G_{\mathbf{h}, \mathbf{x}}\in\mathcal
{T}_{2h_N}$ so that $\mu
_{\mathbf{h}, \mathbf{x}}\in G_{\mathbf{h}, \mathbf{x}}\subset G$,
and $|I(\mu_{\mathbf{h}, \mathbf{x}})-\allowbreak I(\mu)|<\delta$.
\end{lemma}
\begin{pf}
Assume that $0<s_{\mu-}<\infty$. For $0<b<a$ and $f, g$
increasing and left continuous on $[0, a]$, we will use the
notation
\[
I(f, g, a, b)=\frac{\pi^2}{2}\int_0^b \frac{1}{t^2}\di
(f+g)(t)+\frac{\pi^2}{8}\int_b^a \frac{1}{t^2}\di f(t).
\]

By the definition of the topology of $\C{M}$, there is an
$L>s_{\mu-}$ and $U\in\mathcal{T}_L$ neighborhood of $\mu$ such
that $U\subset G$. We can also assume that
$I(f_\mu, g_\mu, L,\break s_{\mu-})>I(\mu)-\gd$.

We can approximate in the Skorokhod topology the
restrictions of $f_\mu, g_\mu$ on $[0, L]$ by monotone left
continuous step functions $f^{(n)}, g^{(n)}$ with finitely
many steps so that $f^{(n)}, g^{(n)}$ are constant on $[L,
\infty)$ and $[s_{\mu-}-1/n, \infty)$ respectively [we use
the left continuity of $g_\mu$ at $s_{\mu-}$ to satisfy that together
with \eqref{I conv}],
they do not have common jump times, and
%
\begin{equation}\label{I conv}
I\bigl(f^{(n)}, g^{(n)}, L, s_{\mu-}\bigr) \to I(f_\mu,g_\mu, L,
s_{\mu-}).
\end{equation}
We can also assume that $f^{(n)}$ has a jump in $(L/2, L)$.

Then we approximate the measure $\mu( \cdot \times
\mathbb R^+)$ on $[0, L]$ by a sequence of measures on $[0,
\infty)$ whose densities are right continuous step functions
$q_n$ with values $0,1$, finitely many steps, $q_n=1$ on
$[s_{\mu-}, \infty)$, and
$q_n=0$ on an interval inside $(s_{\mu-}-2/n, s_{\mu-}-1/n)$.
By introducing extra jumps in $q_n$, we can further
ensure that
%
\begin{equation}\label{supportPoints}
q_n(h)=
\cases{\displaystyle
1 ,&\quad   if $h$ is a jump time of   $f^{(n)} $,\cr\displaystyle
0 ,&\quad   if $h$ is a jump time of   $g^{(n)}$.
}
\end{equation}
If $f_\mu$ has a jump at $s_{\mu-}$, then we require in
addition that
%
\begin{eqnarray} \label{jumpPoint}
f^{(n)} \mbox{ is $1/n$-close to $f_\mu$ at time
$s_{\mu-}-1/n$, }
\mbox{and $q_n=1$ on $[s_{\mu-}-1/n, \infty)$.}\hspace*{-35pt}
\end{eqnarray}
Define the step function
\[
u_n(h)=
\cases{\displaystyle
f^{(n)}(h) ,&\quad  if   $q_n(h)=1 $,\cr\displaystyle
-g^{(n)}(h) ,&\quad  if   $q_n(h) =0$
}
\]
at all points $h\in[0, \infty)$ where $q_n$ does not jump,
and extend it to the remaining finite set of points so that
it is left continuous. Clearly this is a function of the
form $\Phi_{\mathbf{h}, \mathbf{x}}$ with $(\mathbf{h}, \mathbf
{x})\in\C{S}$.

Let $\nu_n=m(u_n)$, the graph occupation measure of $u_n$.
By our construction, $f_{\nu_n}=f^{(n)}$,
$g_{\nu_n}=g^{(n)}$, because of \eqref{supportPoints} and
the right continuity of $q_n$, $\nu_n|[0,L]\times\D{R}\to
\mu|[0,L]\times\D{R}$, $s_{\nu_n+}=\infty$, and
$s_{\nu_n-}\to s_{\mu-}$ because
$s_{\mu-}-2/n<s_{\nu_n-}<s_{\mu-}$.

Furthermore, by \eqref{I conv}, the fact that
$s_{\nu_n-}\to s_{\mu-}$, and that any possible jump of
$f_\mu$ at $s_{\mu-}$ is treated appropriately through
\eqref{jumpPoint}, we have
\[
I(\nu_n)=I\bigl(f^{(n)}, g^{(n)}, L, s_{\nu_n-}\bigr) \to I (f_\mu,
g_\mu, L,s_{\mu-} )\in\bigl(I(\mu)-\gd, I(\mu)\bigr].
\]
Take now $n$ sufficiently large so that $\nu_n\in U, I(\nu_n)>I(\mu
)-\gd$, and let $(\mathbf{h}, \mathbf{x})\in\C{S}$ be such that
$\nu_n=\mu
_{\mathbf{h}, \mathbf{x}}$. Finally, let $G_{\mathbf{h}, \mathbf
{x}}:=U$. Since $f^{(n)}$ has a
jump in $(L/2, L)$, it holds $2h_N>L$, and thus $U\in\mathcal
{T}_L\subset\mathcal{T}_{2 h_N}$.

The remaining truly different cases are $s_{\mu-}=0$,
$s_{\mu-}=s_{\mu+}=\infty$; the proof in these cases is
similar and easier.
\end{pf}

\begin{lemma} \label{darthvader} If $(\mathbf{h}, \mathbf{x})\in
\mathcal S$, and
$\mu_{\mathbf{h}, \mathbf{x}}\in G_{\mathbf{h}, \mathbf{x}}\in
\mathcal{T}_{2h_N}$, then for
all sufficiently small $\eps$, we have
\[
\{m(x_B)\dvtx B\in\mathcal R(\mathbf{h}, \mathbf{x}, \eps, \eps)\}
\subset G_{\mathbf{h}, \mathbf{x}}.
\]
\end{lemma}

\begin{pf}
The paths $B$ contained in $\mathcal R(\mathbf{h}, \mathbf{x}, \eps,
\eps)$ have the property that $x_B$ is a step function
whose jump times and values are close to that of
$\Phi_{\mathbf{h}, \mathbf{x}}$ in the interval $[0, 2 h_N]$; see
\eqref
{hvariation} and
\eqref{xvariation}. In particular, for any $\delta>0$,
there exists $\eps_0>0$ so that for $\eps\in(0, \eps_0)$,
$\{x_B\dvtx B\in\mathcal R(\mathbf{h}, \mathbf{x}, \eps, \eps)\}$ is
contained in the $[0, 2 h_N]$-Skorokhod ball of radius $\delta$ about
$\Phi_{\mathbf{h}, \mathbf{x}}$.

On the space of real left continuous functions on $[0, \infty)$ having
right limits, consider for $L>0$ the topology of Skorokhod convergence
in $[0, L]$. Also let $\C{T}_L'$ the topology of weak convergence on
the compact subsets of $[0, L]\times\D{R}$ for measures on $[0,
\infty)\times\D{R}$. The graph occupation measure is a continuous
functions between the two spaces with the above topologies. Let $G'\in
\C{T}_{2h_N}'$ so that $G'\cap\C{M}=G_{\mathbf{h}, \mathbf{x}}$.
By the
continuity of $m$ just mentioned, it follows that $m^{-1}(G')$ contains
some $[0, 2 h_N]$-Skorokhod ball
around $\Phi_{\mathbf{h}, \mathbf{x}}$, and therefore also the set
$\{x_B\dvtx B\in\mathcal R(\mathbf{h}, \mathbf{x}, \eps, \eps)\}$ for all
$\eps>0$ small enough. Since the image of $\{x_B\dvtx B\in
\mathcal R(\mathbf{h}, \mathbf{x}, \eps, \eps)\}$ under $m$ is also in
$\mathcal M$, the claim follows.~%
\end{pf}

The following lemma is needed in Proposition~\ref{geomSequences2},
toward the proof of the functional law of
iterated logarithm for the environment. In order to show that all rate-1
measures are limit points, we need to show that they can be
approximated by lower-rate ones. This is implied by the
following.

\begin{lemma}\label{l:nomin} The
minimum of the rate function $I$ on an open set is either
zero, infinity, or is not achieved.
\end{lemma}

\begin{pf}
Let $\mu\in G$ with $G$ open and $I(\mu)$ positive and finite.
Recall that $I(\mu)$ is defined in \eqref{rate} in terms of
$f,g$ whose graph $\mu$ is supported on. Let
$\mu_\eps(I\times J)=\mu(I \times(1-\eps)^{-1} J)$, that is, a
scaled version of $\mu$ that is supported on the graph of
$(1-\eps)f$ and $(1-\eps)g$. Then
$I(\mu_\eps)=(1-\eps)I(\mu)$. Also, $\mu_\eps\to\mu$
locally weakly as $\eps\to0$, so for small enough $\eps$
we have $\mu_\eps\in G$.
\end{pf}

\section*{Acknowledgment}
The authors thank Zhan
Shi for bringing this problem to their attention.


%

\printaddresses

\end{document}